\documentclass[10pt]{amsart}
\usepackage{mathtools}
\usepackage[mathscr]{eucal}
\usepackage{ifpdf}
\usepackage{url}
\usepackage{amsmath}
\usepackage{graphicx} \usepackage[all]{xy}
\usepackage{caption, subcaption}
\usepackage[usenames,dvipsnames]{color}
\usepackage{graphicx}
\usepackage[text={5.58in,8.5in},centering,letterpaper,dvips]{geometry}
\usepackage{amsfonts}
\usepackage{epsf}
\usepackage{amssymb}
\usepackage{amsmath}
\usepackage{amscd}
\usepackage{pdfpages}
\usepackage{fancyhdr}
\usepackage{setspace}
\usepackage[all]{xy}
\usepackage{verbatim}
\usepackage{enumerate}
\usepackage{pgf}
\usepackage{comment}
\usepackage[normalem]{ulem}
\usepackage{tikz-cd}
\usepackage[latin1]{inputenc}
\usepackage[colorlinks=true, urlcolor=NavyBlue, linkcolor=NavyBlue, citecolor=NavyBlue]{hyperref}

\theoremstyle{theorem}
\newtheorem{theorem}{Theorem}[section]
\newtheorem{proposition}[theorem]{Proposition}
\newtheorem{lemma}[theorem]{Lemma}
\newtheorem{corollary}[theorem]{Corollary}

\makeatletter
\newtheorem*{rep@theorem}{\rep@title}
\newcommand{\newreptheorem}[2]{%
\newenvironment{rep#1}[1]{%
 \def\rep@title{#2 \ref{##1}}%
 \begin{rep@theorem}}%
 {\end{rep@theorem}}}
\makeatother

\newreptheorem{theorem}{Theorem}
\newreptheorem{lemma}{Lemma}
\newreptheorem{claim}{Claim}
\newreptheorem{question}{Question}
\newreptheorem{corollary}{Corollary}
\newreptheorem{proposition}{Proposition}

\theoremstyle{definition}
\newtheorem{definition}[theorem]{Definition}

\newtheorem{remark}[theorem]{Remark}

\usepackage{accents}
\newlength{\dhatheight}

\author{Ryan Blair, Alexandra Kjuchukova and Ella Pfaff}

\begin{document}

\rhead{\thepage}
\lhead{\author}
\thispagestyle{empty}


\raggedbottom
\pagenumbering{arabic}
\setcounter{section}{0}


\title[Equality between bridge number and meridional rank] 
{Adding a suitable unknot to any link \\equates bridge number and meridional rank}


\maketitle


\begin{abstract}
Given any link $L\subseteq S^3$, we show that it is possible to embed an unknot $U$ in its complement so that the link $L\cup U$ satisfies the Meridional Rank Conjecture (MRC). The bridge numbers in our construction fit into the equality $\beta(L\cup U)=2\beta(L)-1=\text{rank}(\pi_1(S^3\backslash (L\cup U)))$. In addition, we prove the MRC for new infinite families of links and distinguish them from previously settled cases through an application of bridge distance. 
	\end{abstract}

\section{Introduction}

Given a link $L\subset S^3$, we denote its bridge number and meridional rank by $\beta(L)$ and $\mu(L)$, respectively. The meridional rank conjecture, or MRC, due to Cappell and Shaneson~\cite[Problem 1.18]{kirby1995problems}, asks whether the equality  $\beta(L)=\mu(L)$ holds for every link in $S^3$.
Over the last four decades, the conjecture has been established for many links satisfying a variety of geometric conditions~\cite{boileau1985nombre, rost1987meridional, burde1988links, boileau1989orbifold, boileau2017meridionalrank, CH14, Corn14, baader2019coxeter, baader2017symmetric, baader2023bridge, Dutra22}. The analogous statement is also shown to hold for some knotted spheres in $S^4$~\cite{joseph2024meridional}. One shared characteristic of the above diverse families is that all knots among them have  low distance relative to their bridge number; specifically, for any knot covered by the results mentioned above, the ratio of bridge distance over bridge number is less than 3. In Proposition~\ref{thm:fishnet-quotients} and in Section~\ref{sec:other-good-fishnets}, we give new families of links which satisfy MRC and which contain knots for which the bridge number and the ratio of bridge distance to bridge number can be arbitrarily high. 
The links we examine in this paper, which we call {fishnets} (see Definition~\ref{def:fishnet} and Figure~\ref{Fig:JohnsonMoriah}),  can also be used to prove the following.

\begin{theorem} \label{thm:sublink}  For any link $L\subseteq S^3$ in bridge position with $b>1$ bridges, there is an unknot  $U\subseteq S^3\backslash L$ such that 
\begin{equation}\label{eq:main}
    \beta(L\cup U)=\mu(L\cup U)=2b-1=\text{rank}(\pi_1(S^3\backslash (L\cup U))).
\end{equation}
\end{theorem}

Given a link $L$ in bridge position with $b$ bridges, we describe an embedding of an unknot $U$ in the complement of $L$ so that Equation~\ref{eq:main} holds. In our construction, a minimal generating set of meridians for $L\cup U$ contains exactly $b$ meridians of $L$. Moreover, it is in fact a minimal generating set for the group, without restricting the conjugacy classes of generators. Note that we allow $b$ to exceed $\beta(L)$, in which case the number of meridians of $L$ needed in a generating set for $L\cup U$ exceeds $\beta(L)$.  

The proof of Theorem~\ref{thm:sublink} relies on finding maximal rank Coxeter quotients, in the sense studied in~\cite{baader2019coxeter, baader2023bridge, blair2024coxeter}, of fundamental groups of fishnet link complements. A map from the group of a link $L$ onto a Coxeter group $G(\Gamma)$ is maximal rank if it maps meridians to reflections and the Coxeter rank of $G(\Gamma)$ is equal to (any upper bound on) the bridge number of $L$. We define fishnet links as plat closures of braids which admit diagrams as in Figure~\ref{Fig:JohnsonMoriah}. The bridge distance of what we call strong fishnet links was previously studied by Johnson and Moriah~\cite{johnson2016bridge}, and incompressible surfaces in the complement of fishnets was investigated by Finkelstein and Moriah~\cite{finkelstein1999closed, FM99}. 
By definition, a fishnet link $L$ is determined by a set of integer parameters corresponding to the powers of the standard braid generators in a fishnet diagram of $L$. In Proposition~\ref{thm:fishnet-quotients}, we give sufficient conditions, in terms of this parameter set, for a fishnet link $L$ to admit a maximal rank Coxeter quotient. In particular, we find infinite families of fishnets which satisfy MRC, with the properties discussed above.

\begin{definition}\label{def:fishnet} Given an integer valued vector 
\[
    \mathfrak{t}:= (t_{1,2},t_{1,4}, \dots, t_{1,m-2};t_{2,1},t_{2,3},\dots,t_{2,m-1}; \dots;t_{n,2},t_{n,4}, \dots, t_{n,m-2}),
    \] 
    let $L_\mathfrak{t}$ denote the plat closure of the $m$-stranded braid 
\[
\zeta_\mathfrak{t} := (\sigma_2^{t_{1,2}}\sigma_4^{t_{1,4}}\dots \sigma_{m-2}^{t_{1,m-2}})(\sigma_1^{t_{2,1}}\sigma_3^{t_{2,3}}\dots \sigma_{m-1}^{t_{2,m-1}})(\sigma_2^{t_{3,2}}\sigma_4^{t_{3,4}}\dots \sigma_{m-2}^{t_{3,m-2}}) \dots (\sigma_2^{t_{n,2}}\sigma_4^{t_{n,4}}\dots \sigma_{m-2}^{t_{n,m-2}}).
\]
(See Figure~\ref{Fig:JohnsonMoriah}.) 
    If all the integers $t_{i, j}$ are non-zero, we say $L$ is a {\it regular fishnet link}, or simply a {\it fishnet}, of width $m$ and height $n$. When $|t_{i, j}|\geq3$ for all $i, j,$ we say $L$ is a {\it strong fishnet}. If some of the parameters $t_{i, j}$ are 0, we say $L$ is a {\it loose fishnet}.
    \end{definition}

\begin{remark}
    By allowing sufficiently many $t_{i, j}$ to be zero, we can present any link $L$ as a loose fishnet with width equal to $2\beta(L)$. In other words, one can regard the symmetry imposed in fishnet diagrams as merely a convenient organizational principle applicable to all links.
\end{remark}
 
We will refer to the collection of twist regions whose parameters are of the form $t_{i, \ast}$ as {\it row $i$} of the fishnet; and to the the twist regions whose parameters are of the form $t_{\ast, j}$ as {\it column $j$}. Note that the row indices in each column have the same parity; that the integer $m$ is always even while $n$ is always odd; and that a fishnet diagram is by definition in bridge position with $\frac{m}{2}$ bridges. In what follows, it will be helpful to name the greatest common divisors of parameters which share a column, so we
let $d_1, d_2, \dots, d_{m{-}1}$ be the positive integers defined as follows:
\begin{equation}\label{eq:d_js}
    d_j := \gcd_{\substack{1\leq i\leq n_j\\ i\equiv j+1\mod 2}}
\left\{t_{i,j}\right\},
\end{equation} 
where $n_j=n-1$ if $j$ is odd and $n_j=n$ if $j$ is even. If for some $j=k$ we have that all parameters $t_{i,k}$ in column $k$ are equal to zero, we use the convention that $d_k=\infty$.

In the next theorem we use Coxeter quotients of link groups to derive sufficient conditions on the parameter vector $\mathfrak{t}$  which guarantee that the meridional rank of $L_\mathfrak{t}$ is also equal to $\frac{m}{2}$.  As remarked above, these fishnets provide the first evidence that links which satisfy MRC are not limited to links which have bounded bridge number, bounded bridge distance, or bounded ratio of bridge distance over bridge number. See Theorem~\ref{thm:low-d} for a discussion of these complexity measures for previously known cases.

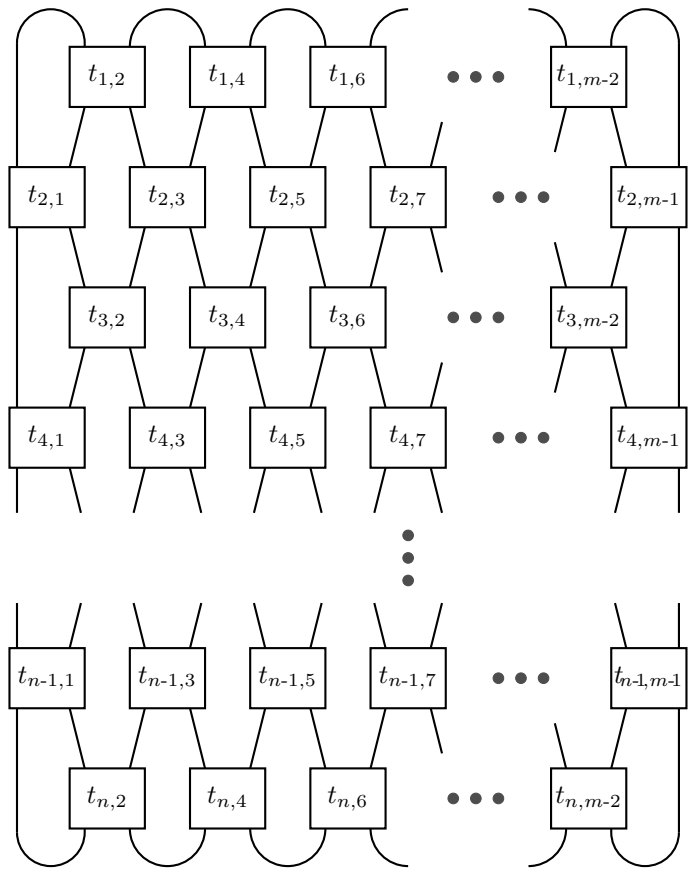
\begin{figure}[ht]
\centering

\begin{tikzpicture}
        [x=1mm,y=1mm,
        box/.style={rectangle, inner sep=0mm, draw=black, fill=white, thick, minimum width=10mm, minimum height=8mm}]

        \node[box](tn2) at (0,0){$t_{n,2}$};
        \node[box](tn4) at (16,0){$t_{n,4}$};
        \node[box](tn6) at (32,0){$t_{n,6}$};
        \node[box](tnm2) at (64,0){$t_{n,m\text-2}$};

        \node[box](tn11) at (-8,16){$t_{n\text-1,1}$};
        \node[box](tn13) at (8,16){$t_{n\text-1,3}$};
        \node[box](tn15) at (24,16){$t_{n\text-1,5}$};
        \node[box](tn17) at (40,16){$t_{n\text-1,7}$};
        \node[box](tn1m1) at (72,16){$t\!_{n\text-\! 1\!,m\text-\!1}$};

        \node[box](t41) at (-8,48){$t_{4,1}$};
        \node[box](t43) at (8,48){$t_{4,3}$};
        \node[box](t45) at (24,48){$t_{4,5}$};
        \node[box](t47) at (40,48){$t_{4,7}$};
        \node[box](t4m1) at (72,48){$t_{4,m\text-1}$};

        \node[box](t32) at (0,64){$t_{3,2}$};
        \node[box](t34) at (16,64){$t_{3,4}$};
        \node[box](t36) at (32,64){$t_{3,6}$};
        \node[box](t3m2) at (64,64){$t_{3,m\text-2}$};

        \node[box](t21) at (-8,80){$t_{2,1}$};
        \node[box](t23) at (8,80){$t_{2,3}$};
        \node[box](t25) at (24,80){$t_{2,5}$};
        \node[box](t27) at (40,80){$t_{2,7}$};
        \node[box](t2m1) at (72,80){$t_{2,m\text-1}$};

        \node[box](t12) at (0,96){$t_{1,2}$};
        \node[box](t14) at (16,96){$t_{1,4}$};
        \node[box](t16) at (32,96){$t_{1,6}$};
        \node[box](t1m2) at (64,96){$t_{1,m\text-2}$};
        
        \draw[thick] 
        (-12,-4.5) -- (-12,12)
        (-3,4) -- (-5,12) 
        (3,4) -- (5,12) 
        (13,4) -- (11,12) 
        (19,4) -- (21,12) 
        (29,4) -- (27,12) 
        (35,4) -- (37,12) 
        (44.5,6) -- (43,12)
        (61,4) -- (59.5,10)
        (67,4) -- (69,12)
        (76,-4.5) -- (76,12);
        \draw[thick] 
        (-12,20) -- (-12,26)%
        (-5,20) -- (-3.5,26)%
        (5,20) -- (3.5,26)%
        (11,20) -- (12.5,26)%
        (21,20) -- (19.5,26)%
        (27,20) -- (28.5,26)%
        (37,20) -- (35.5,26)%
        (43,20) -- (44.5,26)%
        (69,20) -- (67.5,26)%
        (76,20) -- (76,26);%
        \draw[thick] 
        (-12,38) -- (-12,44)%
        (-3.5,38) -- (-5,44)%
        (3.5,38) -- (5,44)%
        (12.5,38) -- (11,44)%
        (19.5,38) -- (21,44)%
        (28.5,38) -- (27,44)%
        (35.5,38) -- (37,44)%
        (44.5,38) -- (43,44)%
        (67.5,38) -- (69,44)%
        (76,38) -- (76,44);%
        \draw[thick] 
        (-12,52) -- (-12,76)
        (-5,52) -- (-3,60) 
        (5,52) -- (3,60) 
        (11,52) -- (13,60) 
        (21,52) -- (19,60) 
        (27,52) -- (29,60) 
        (37,52) -- (35,60) 
        (43,52) -- (44.5,58)
        (59.5,54) -- (61,60)
        (69,52) -- (67,60)
        (76,52) -- (76,76);
        \draw[thick] 
        (-3,68) -- (-5,76) 
        (3,68) -- (5,76) 
        (13,68) -- (11,76) 
        (19,68) -- (21,76) 
        (29,68) -- (27,76) 
        (35,68) -- (37,76)
        (44.5,70) -- (43,76)
        (61,68) -- (59.5,74)
        (67,68) -- (69,76);
        \draw[thick] 
        (-12,84) -- (-12,100.5)
        (-5,84) -- (-3,92) 
        (5,84) -- (3,92) 
        (11,84) -- (13,92) 
        (21,84) -- (19,92) 
        (27,84) -- (29,92) 
        (37,84) -- (35,92) 
        (43,84) -- (44.5,90)
        (59.5,86) -- (61,92)
        (69,84) -- (67,92)
        (76,84) -- (76,100.5);

        \draw[thick]
        (-3,-4) -- (-3,-4.5)

        (-12,-4.5) arc (180:360:4.5)
        (3,-4) arc (180:360:5)
        (19,-4) arc (180:360:5)
        (35,-4) arc (180:270:5)
        (61,-4) arc (0:-90:5)
        (67,-4.5) arc (180:360:4.5)

        (67,-4) -- (67,-4.5);

        \draw[thick] 
        (-3,100) -- (-3,100.5)

        (-12,100.5) arc (180:0:4.5)
        (3,100) arc (180:0:5)
        (19,100) arc (180:0:5)
        (35,100) arc (180:90:5)
        (61,100) arc (0:90:5)
        (67,100.5) arc (180:0:4.5)

        (67,100) -- (67,100.5);
        
        \filldraw[black!70] 
        (49,0) circle (2pt)
        (46,0) circle (2pt)
        (52,0) circle (2pt)

        (55,16) circle (2pt)
        (52,16) circle (2pt)
        (58,16) circle (2pt)
        
        (40,32) circle (2pt)
        (40,29) circle (2pt)
        (40,35) circle (2pt)
        
        (55,48) circle (2pt)
        (52,48) circle (2pt)
        (58,48) circle (2pt)
        
        (49,64) circle (2pt)
        (46,64) circle (2pt)
        (52,64) circle (2pt)

        (55,80) circle (2pt)
        (52,80) circle (2pt)
        (58,80) circle (2pt)
        
        (49,96) circle (2pt)
        (46,96) circle (2pt)
        (52,96) circle (2pt);     
    \end{tikzpicture}

\caption{A fishnet link with width $m$ and height $n$, determined by the parameter vector $\mathfrak{t}:= (t_{1,2},t_{1,4}, \dots, t_{1,m-2};t_{2,1},t_{2,3},\dots,t_{2,m-1}; \dots;t_{n,2},t_{n,4}, \dots, t_{n,m-2})$.}
\label{Fig:JohnsonMoriah}
\end{figure}

\begin{proposition}\label{thm:fishnet-quotients}
    Let $L_\mathfrak{t}$ denote the regular fishnet link determined by the parameter vector 
    \[
    \mathfrak{t}:= (t_{1,2},t_{1,4}, \dots, t_{1,m-2};t_{2,1},t_{2,3},\dots,t_{2,m-1}; \dots;t_{n,2},t_{n,4}, \dots, t_{n,m-2});
    \] 
    and let $d_1, d_2, \dots, d_{m{-}1}$ be as defined in Equation~\ref{eq:d_js}. 
If  
\[
d_{2j}>1\text{ for all } 1\leq j\leq \frac{m-2}{2}, 
\]
then $\mu(L)=\beta(L)= \frac{m}{2}$. In particular, the MRC holds for $L$. Moreover, if $L$ is a strong fishnet (that is, $|t_{i,j}|\geq 3, \forall i, j$) and $m\geq 6,$ then $L$ has bridge distance  $\lceil n /(m - 4))  \rceil$. 
\end{proposition}

\begin{corollary}\label{cor:knots-too}
    For any $b\in \mathbb{N}_{\geq 3}$ and $r\in \mathbb{R}$, there exists a knot $K$ with $\beta(K)=\mu(K)=b$ such that $\frac{d}{b}> r,$ where $d$ is the bridge distance of $K.$
\end{corollary}

The proof of Theorem \ref{thm:fishnet-quotients}, given in Section~\ref{sec:quotients-proof}, is a straightforward application of the technique introduced in~\cite{baader2019coxeter} and the main result in~\cite{johnson2016bridge}. We also note that the above list of conditions is far from exhaustive. Other families of fishnets which satisfy MRC are given in Section~\ref{sec:other-good-fishnets}. These generalize in obvious ways to produce additional examples with arbitrarily large bridge number and bridge distance. We have singled out the families in Proposition~\ref{thm:fishnet-quotients} only because these cases lend themselves to a particularly succinct description and proof; and because they are the ones needed in the proof of Theorem~\ref{thm:sublink}. 

The rest of the paper is organized as follows.  In Section~\ref{sec:fishnets}, we prove Proposition~\ref{thm:fishnet-quotients} and we introduce other families of fishnets which admit maximal rank Coxeter quotients. Section~\ref{sec:sublink-thm-proof} contains 
the proof of Theorem~\ref{thm:sublink}. In Section~\ref{sec:distance}, we review the definition of bridge distance and prove Theorem~\ref{thm:low-d}, which states that knots for which MRC was previously established have either bounded bridge number, bounded distance, or bounded ratio of bridge distance over bridge number. 
Section~\ref{sec:limerick} contains a brief conclusion.

\section{Maximal Rank Coxeter Quotients of Fishnets} \label{sec:fishnets}
In this section, we prove Proposition~\ref{thm:fishnet-quotients} and Corollary~\ref{cor:knots-too}; and we give some additional families of fishnet links which satisfy MRC. We also lay the groundwork for the proof of Theorem~\ref{thm:sublink}.

\subsection{Maximal Rank Coxeter Quotients} \label{sec:quotients-proof}
Our technique throughout will be the following: given $L_\mathfrak{t}$ a loose fishnet of width $m$, we will construct a quotient $\varphi: \pi_1(S^3\setminus L_\mathfrak{t}) \rightarrow G_\mathfrak{t},$ where $G_\mathfrak{t}$ is a Coxeter group of Coxeter rank $\frac{m}{2}$. As the notation suggests, the group $G_\mathfrak{t}$  will be determined by the parameter vector $\mathfrak{t}$ which defines the fishnet. 
Note that $\varphi$ will be a maximal rank Coxeter quotient in the sense of~\cite{blair2024coxeter}; in particular, $\varphi$ maps meridians of $L_\mathfrak{t}$ to reflections of $G_\mathfrak{t}$. 
Then, the Coxeter rank of $G_\mathfrak{t}$ is equal to the number of local maxima of $L_\mathfrak{t}$ depicted in Figure~\ref{Fig:JohnsonMoriah}. Such a quotient gives a tight lower bound on the meridional rank of $\pi_1(S^3 \backslash L_\mathfrak{t})$, as in~\cite{baader2019coxeter, baader2023bridge}, proving that $\beta(L)=\mu(L)=\frac{m}{2}$ as claimed. 

We will represent quotients $\varphi: \pi_1(S^3\backslash L_\mathfrak{t})\to G_\mathfrak{t}$ in a diagram $D$ of $L_\mathfrak{t}$ in the usual way, going back to Fox: by labeling strands in $D$ with the images of  the corresponding Wirtinger meridians in $G_\mathfrak{t}$. This labeling defines a quotient map iff at each crossing the image of the Wirtinger relation holds in $G_\mathfrak{t}$. Given a twist region in $D$, to determine the image of all strands in this region under $\varphi,$ it suffices to label the two strands incoming into the twist region from the top (resp. the bottom). 

 The following observations will be used repeatedly in the paper. Given any twist region $\sigma_{k}^{t_{j,k}}$ in the fishnet $L_\mathfrak{t}$, let 
$a$ and $b$ denote the labels of the two incoming strands from the top and let $r$ denote the order of $(ab)$ in $G$.
If $r$ divides $t_{j,k},$ then the Wirtinger meridians of the two outgoing strands are also labeled by $a$ and $b$; moreover, $a$ and $b$ appear in the same order, left to right, at the top and the bottom of the twist region. If $r=2$ (equivalently, $a$ and $b$ commute), then, regardless of whether $r$ divides $t_{j, k}$, the outgoing strands are again labeled $a$ and $b$; the order in which the labels $a, b$ appear is determined by the parity of $t_{j, k}$ in the obvious way. 
These diagrammatic constructions were inspired by Brunner~\cite{brunner1992geometric}, who studied Artin quotients of link groups. 

\begin{proof}[Proof of Proposition~\ref{thm:fishnet-quotients}]
Consider the Coxeter group 
\[
G_\mathfrak{t}=\left\langle a_1, a_2, \dots, a_{\frac{m}{2}}\ \big\lvert \ a_1^2=a_2^2\dots = a_{\frac{m}{2}}^2=1 \text{ and } (a_ja_{j+1})^{d_{2j}}=1 \text{ for all } j \in\{1,...,\tfrac{m}{2}{-}1\} \right\rangle.
\]
Since $d_j>1$ for all $j,$ it follows that $G_\mathfrak{t}$ has Coxeter rank equal to $\frac{m}{2}$~\cite[Lemma~2.1]{felikson2010reflection}. Let  $L_\mathfrak{t}$ be represented in the fishnet diagram determined by the parameter vector  $\mathfrak{t}$ in the sense of Definition~\ref{def:fishnet}.
We begin by labeling the strands containing the $\frac{m}{2}$ local maxima in the diagram by $a_1, a_2, \dots, a_{\frac{m}{2}},$ in this order. Recall that, by assumption, for each $j$, $d_j$ divides $t_{i,j} \ \forall i$.  Therefore, by the discussion above, for each twist region in the diagram, the labels of the incoming strands are the same as those of the outgoing strands. This implies that the induced labels at the two endpoints of a strand containing a local minimum agree. We have therefore arrived at a coherent coloring of the diagram by elements of the group  $G_\mathfrak{t}$. This defines the desired quotient $\varphi: \pi_1(S^3\backslash L_\mathfrak{t})\to G_\mathfrak{t}$ of Coxeter rank $\frac{m}{2}$, proving the claim that $\beta(L_\mathfrak{t})=\mu(L_\mathfrak{t})=\frac{m}{2}$.

To prove the final claim of the theorem, assume that the parameter vector $\mathfrak{t}$ satisfies $|t_{i,j}|\geq 3 \ \forall i,j$. Then, $L_\mathfrak{t}$  fulfills the hypotheses of Theorem~\ref{thm:JohnsonMoriah}. Therefore, the distance of the induced bridge sphere in the fishnet diagram is $\lceil n /(m - 4) \rceil$. 
\end{proof}

\begin{proof}[Proof of Corollary~\ref{cor:knots-too}] 
    Fix $r\in \mathbb{R}_{>0}$ and $b\in \mathbb{N}_{\geq 3}$. Let $L_\mathfrak{t}$ be the fishnet determined by the parameter vector
    \[ 
    \mathfrak{t}= (t_{1,2},t_{1,4}, \dots, t_{1,m-2}; \dots, t_{n,m{-}2}),
    \]
    where $\frac{m}{2}=b\geq 3$ and where we choose $n$ such that $n\geq \frac{rm^2}{2}.$ In this construction, we may select the integer parameters freely, so we assume that $\forall i,j$, we have $t_{i,j}>3$ and $d_{2j}>1$. By Proposition~\ref{thm:fishnet-quotients}, it follows that $\beta(L_\mathfrak{t})=\frac{m}{2}$ and that $d(L_\mathfrak{t})\geq \frac{n}{m-4}$. Thus, $\frac{d}{b}\geq \frac{n}{m-4}\cdot\frac{2}{m} >r,$ as claimed. To conclude the proof, note that by Lemma~\ref{lem:comp-number} we may furthermore ensure that $L_\mathfrak{t}$ is a knot.
\end{proof}
\subsection{Other Fishnets which satisfy MRC}\label{sec:other-good-fishnets}

\begin{figure}[ht]
    \centering

    \begin{tikzpicture}
        [x=1mm,y=1mm,
        box/.style={rectangle, inner sep=0mm, draw=black!75, fill=white, thick, minimum width=3mm, minimum height=3mm},
        box2/.style={rectangle, inner sep=0mm, draw=black!15, fill=white, thick, minimum width=3mm, minimum height=3mm}]
        \definecolor{comp1}{HTML}{BE0032}
        \definecolor{comp2}{HTML}{191970}
        \definecolor{comp3}{HTML}{177245}
        \colorlet{comp4}{black!15}
        
        \draw[comp1, thick]
        (0.5,0) -- (4.5,10)
        (9.5,0) -- (5.5,10)
        (19.5,0) -- (15.5,10)
        (29.5,0) -- (25.5,10)
        (40.5,0) -- (44.5,10)
        ;
        \draw[comp1, thick]
        (4.5,10) -- (0.5,20)
        (5.5,10) -- (9.5,20)
        (15.5,10) -- (19.5,20)
        (25.5,10) -- (29.5,20)
        (45.5,10) -- (49.5,20)
        ;
        \draw[comp1, thick]
        (0.5,20) -- (4.5,30)
        (9.5,20) -- (5.5,30)
        (20.5,20) -- (24.5,30)
        (30.5,20) -- (34.5,30)
        (49.5,20) -- (45.5,30)
        ;
        \draw[comp1, thick]
        (-4.5,30) -- (-0.5,40)
        (4.5,30) -- (0.5,40)
        (5.5,30) -- (9.5,40)
        (25.5,30) -- (29.5,40)
        (34.5,30) -- (30.5,40)
        (44.5,30) -- (40.5,40)
        ;
        \draw[comp1, thick]
        (-0.5,40) -- (-4.5,50)
        (0.5,40) -- (4.5,50)
        (10.5,40) -- (14.5,50)
        (29.5,40) -- (25.5,50)
        (30.5,40) -- (34.5,50)
        (39.5,40) -- (35.5,50)
        ;
        \draw[comp1, thick]
        (4.5,50) -- (0.5,60)
        (15.5,50) -- (19.5,60)
        (25.5,50) -- (29.5,60)
        (34.5,50) -- (30.5,60)
        (35.5,50) -- (39.5,60)
        ;
        \draw[comp1, thick]
        (-6.25,0) -- (-6.25,10)
        (-6.25,10) -- (-6.25,30)
        (-6.25,50) -- (-6.25,60)
        ;
        \draw[comp1, thick]
        (-6.25,0) .. controls (-6.25,-7.5) and (-0.5,-7.5) .. (-0.5,0)
        (10.5,0) .. controls (10.5,-7.5) and (19.5,-7.5) .. (19.5,0)
        (30.5,0) .. controls (30.5,-7.5) and (39.5,-7.5) .. (39.5,0)
        ;
        \draw[comp1, thick]
        (-6.25,60) .. controls (-6.25,67.5) and (-0.5,67.5) .. (-0.5,60)
        (20.5,60) .. controls (20.5,67.5) and (29.5,67.5) .. (29.5,60)
        (30.5,60) .. controls (30.5,67.5) and (39.5,67.5) .. (39.5,60)
        ;

        \draw[comp2, thick]
        (50.5,0) -- (54.5,10)
        (59.5,0) -- (55.5,10)
        (60.5,0) -- (64.5,10)
        ;
        \draw[comp2, thick]
        (54.5,10) -- (50.5,20)
        (55.5,10) -- (59.5,20)
        (64.5,10) -- (60.5,20)
        ;
        \draw[comp2, thick]
        (50.5,20) -- (54.5,30)
        (59.5,20) -- (55.5,30)
        (60.5,20) -- (64.5,30)
        ;
        \draw[comp2, thick]
        (54.5,30) -- (50.5,40)
        (55.5,30) -- (59.5,40)
        (64.5,30) -- (60.5,40)
        ;
        \draw[comp2, thick]
        (49.5,40) -- (45.5,50)
        (59.5,40) -- (55.5,50)
        (60.5,40) -- (64.5,50)
        ;
        \draw[comp2, thick]
        (45.5,50) -- (49.5,60)
        (55.5,50) -- (59.5,60)
        (64.5,50) -- (60.5,60)
        ;
        \draw[comp2, thick]
        (66.25,0) -- (66.25,10)
        (66.25,10) -- (66.25,30)
        (66.25,30) -- (66.25,50)
        (66.25,50) -- (66.25,60)
        ;
        \draw[comp2, thick]
        (50.5,0) .. controls (50.5,-7.5) and (59.5,-7.5) .. (59.5,0)
        (60.5,0) .. controls (60.5,-7.5) and (66.25, -7.5) .. (66.25,0)
        ;
        \draw[comp2, thick]
        (50.5,60) .. controls (50.5,67.5) and (59.5,67.5) .. (59.5,60)
        (60.5,60) .. controls (60.5,67.5) and (66.25, 67.5) .. (66.25,60)
        ;

        \draw[comp3, thick]
        (-0.5,0) -- (-4.5,10)
        (10.5,0) -- (14.5,10)
        ;
        \draw[comp3, thick]
        (-4.5,10) -- (-0.5,20)
        (14.5,10) -- (10.5,20)
        ;
        \draw[comp3, thick]
        (-0.5,20) -- (-4.5,30)
        (10.5,20) -- (14.5,30)
        ;
        \draw[comp3, thick]
        (15.5,30) -- (19.5,40)
        ;
        \draw[comp3, thick]
        (19.5,40) -- (15.5,50)
        ;
        \draw[comp3, thick]
        (-4.5,50) -- (-0.5,60)
        (14.5,50) -- (10.5,60)
        ;
        \draw[comp3, thick]
        (-6.25,30) -- (-6.25,50)
        ;
        \draw[comp3, thick]
        (0.5,0) .. controls (0.5,-7.5) and (9.5,-7.5) .. (9.5,0)
        ;
        \draw[comp3, thick]
        (0.5,60) .. controls (0.5,67.5) and (9.5,67.5) .. (9.5,60)
        ;

        \draw[comp4, thick]
        (20.5,0) -- (24.5,10)
        (30.5,0) -- (34.5,10)
        (39.5,0) -- (35.5,10)
        (49.5,0) -- (45.5,10)
        ;
        \draw[comp4, thick]
        (24.5,10) -- (20.5,20)
        (34.5,10) -- (30.5,20)
        (35.5,10) -- (39.5,20)
        (44.5,10) -- (40.5,20)
        ;
        \draw[comp4, thick]
        (19.5,20) -- (15.5,30)
        (29.5,20) -- (25.5,30)
        (39.5,20) -- (35.5,30)
        (40.5,20) -- (44.5,30)
        ;
        \draw[comp4, thick]
        (14.5,30) -- (10.5,40)
        (24.5,30) -- (20.5,40)
        (35.5,30) -- (39.5,40)
        (45.5,30) -- (49.5,40)
        ;
        \draw[comp4, thick]
        (9.5,40) -- (5.5,50)
        (20.5,40) -- (24.5,50)
        (40.5,40) -- (44.5,50)
        (50.5,40) -- (54.5,50)
        ;
        \draw[comp4, thick]
        (5.5,50) -- (9.5,60)
        (24.5,50) -- (20.5,60)
        (44.5,50) -- (40.5,60)
        (54.5,50) -- (50.5,60)
        ;

        \draw[comp4, thick]
        (20.5,0) .. controls (20.5,-7.5) and (29.5,-7.5) .. (29.5,0)
        (40.5,0) .. controls (40.5,-7.5) and (49.5,-7.5) .. (49.5,0)
        ;
        \draw[comp4, thick]
        (10.5,60) .. controls (10.5,67.5) and (19.5,67.5) .. (19.5,60)
        (40.5,60) .. controls (40.5,67.5) and (49.5,67.5) .. (49.5,60)
        ;

        \node[box](72) at (0,0){\textcolor{black!50}{$\mathsf{3}$}};
        \node[box](74) at (10,0){};
        \node[box](76) at (20,0){};
        \node[box](78) at (30,0){};
        \node[box](710) at (40,0){};
        \node[box](712) at (50,0){};
        \node[box](714) at (60,0){};

        \node[box](61) at (-5,10){};
        \node[box](63) at (5,10){};
        \node[box](65) at (15,10){};
        \node[box](67) at (25,10){};
        \node[box2](69) at (35,10){\textcolor{black!15}{$\mathsf{1\!1}$}}; 
        \node[box](611) at (45,10){};
        \node[box](613) at (55,10){};
        \node[box](615) at (65,10){};

        \node[box](52) at (0,20){\textcolor{black!50}{$\mathsf{6}$}};
        \node[box](54) at (10,20){\textcolor{black!50}{$\mathsf{1\!0}$}};
        \node[box](56) at (20,20){};
        \node[box](58) at (30,20){};
        \node[box2](510) at (40,20){};
        \node[box](512) at (50,20){\textcolor{black!50}{$\mathsf{1\!3}$}};
        \node[box](514) at (60,20){};

        \node[box](41) at (-5,30){};
        \node[box](43) at (5,30){\textcolor{black!50}{$\mathsf{1}$}};
        \node[box](45) at (15,30){};
        \node[box](47) at (25,30){};
        \node[box](49) at (35,30){};
        \node[box](411) at (45,30){};
        \node[box](413) at (55,30){};
        \node[box](415) at (65,30){};

        \node[box](32) at (0,40){\textcolor{black!50}{$\mathsf{3}$}};
        \node[box](34) at (10,40){};
        \node[box](36) at (20,40){};
        \node[box](38) at (30,40){\textcolor{black!50}{$\mathsf{7}$}};
        \node[box](310) at (40,40){};
        \node[box](312) at (50,40){};
        \node[box](314) at (60,40){};

        \node[box](21) at (-5,50){};
        \node[box](23) at (5,50){};
        \node[box](25) at (15,50){\textcolor{black!50}{$\mathsf{5}$}};
        \node[box](27) at (25,50){};
        \node[box](29) at (35,50){};
        \node[box](211) at (45,50){};
        \node[box](213) at (55,50){};
        \node[box](215) at (65,50){};

        \node[box](12) at (0,60){\textcolor{black!50}{$\mathsf{3}$}};
        \node[box](14) at (10,60){};
        \node[box](16) at (20,60){};
        \node[box](18) at (30,60){\textcolor{black!50}{$\mathsf{7}$}};
        \node[box](110) at (40,60){};
        \node[box](112) at (50,60){};
        \node[box](114) at (60,60){};
        
    \end{tikzpicture}
    
    \caption{A fishnet $M$ which satisfies MRC without fulfilling the hypotheses of Proposition~\ref{thm:fishnet-quotients}. $M$ is the link of a green unknot, a red and blue sublink in $5$-bridge position, and a $(2, 11)$-torus knot, faintly pictured. Unnumbered twist-regions which involve two colors are determined mod~$2$ as dictated by the colors. Omitted parameters in monochromatic twist regions are arbitrary, subject to the condition that the last three columns determine a non-trivial blue summand. }
    \label{fig:monster-fishnet}
\end{figure}
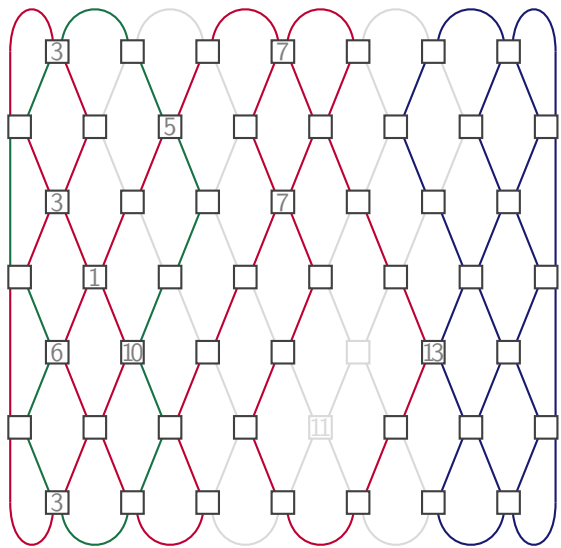

In what follows, we will work in the setting of loose fishnets. 
We will see in this section many new families of fishnets which admit maximal rank Coxeter quotients and thus satisfy MRC. We encourage the reader to understand these families as extensions of representative examples. 

\subsubsection{Decomposing Fishnets} 
The key observation is that the existence of a quotient for a fishnet $L_\mathfrak{t}$ can be detected by decomposing $L_\mathfrak{t}$ into a union of sublinks, $L_\mathfrak{t}=\bigcup_i L_i$. In particular, the $L_i$ are not necessarily knots. 

Consider the fishnet $M$ given by Figure~\ref{fig:monster-fishnet}. With the majority of boxes at most determined up to parity, $M$ does not necessarily satisfy the gcd condition of Proposition~\ref{thm:fishnet-quotients}. However, $M$ admits a maximal rank Coxeter quotient. For fishnets of more than one component, it is common for the existence of a quotient to be dependent on relatively few parameters, as is the case for $M$.

The colors in Figure~\ref{fig:monster-fishnet} are included to guide the construction of the desired quotient. First, the faintly pictured sublink (colored so for diagrammatic clarity) is a 2-bridge link and therefore has a maximal rank Coxeter quotient given by a dihedral group. In this case, the quotient group is $D_{11},$ the dihedral group of order 22. Next, consider the red/blue sublink. Because the red meets the blue in a single twist region, we regard each as a summand. If both summands admit maximal rank Coxeter quotients, then the red/blue sublink does as well. By assumption, the blue is a 2-bridge link summand, hence has some dihedral quotient, $D_b$. 
The red is a 3-bridge link summand, but due to the twist region parameterized by $t_{4,2}=1$, it fails the gcd condition given in Proposition~\ref{thm:fishnet-quotients} and does not have an obvious quotient. Instead, we consider the red summand along with the green unknot. Together these have 4-bridges, and in fact satisfy the gcd condition needed for the red/green to admit a maximal rank Coxeter quotient, $C_{r,g}$. Moreover, the green does not meet the blue in any twist region, so we can regard the red/green and the blue as two summands of the red/green/blue link, summands which both admit maximal rank quotients. By reconciling these two quotient groups, we construct a maximal rank Coxeter quotient, $C_{r,g,b}$, for the red/green/blue sublink. Acting by conjugation on $C_{r,g}$ and $D_b$ ensures that two Coxeter generators, $r$ and $b$, from $C_{r,g}$ and $D_b$ respectively, label the strands entering the top of twist region $t_{5,12}=13$. Then, a minimal Coxeter generating set is given by the union of minimal Coxeter generating sets of $C_{r,g}$ and $D_{b}$; the relations are given by the union of relations for these two groups, together with the additional relation $(rb)^{13}=1$. At last, this results in the maximal rank Coxeter quotient for $M$, $C_M = D_{11} \oplus C_{r,g,b}$.

Note that none of the components of the link in the above example is presented in a diagram of a regular fishnet (not even the blue summand). Moreover, the example illustrates the following more general phenomena:
\begin{enumerate}[(1)]
    \item \label{obv: direct sum} If $k$ disjoint sublinks of a link $L= L_1 \cup \dots \cup L_k$ independently admit maximal rank\footnote{In this context, ``maximal" is meant in the following sense: the Coxeter rank of the quotient of $L_i$ equals an upper bound on the bridge number of $L_i$ as a \textit{sublink} of $L$, that is, the number of bridges that $L_i$ contributes to $L$.} Coxeter quotients, $G_1, \dots, G_k$, then $L$ admits a maximal rank quotient onto the direct sum $G:=G_1\oplus \dots \oplus G_k$. The existence of a quotient $\pi_1(S^3\backslash L)\twoheadrightarrow G$ is not affected by the way these sublinks link each other. In particular, they do not need to form a regular fishnet.
    \item  Our ability to detect the existence of a quotient in a fixed diagram may depend on the way we subdivide a link into a union of sublinks. For instance, if some component $K_s$ of a link $L$ does not admit a maximal rank Coxeter quotient, it may still be the case that the link $K_s \cup K_t$ admits one, where $K_t$ is some other component of $L$ which meets $K_s$ in at least one twist region.
    \item Let $L_1$ and $L_2$ form a split link,  presented in a diagram such that $L_1$ and $L_2$ are contained in disjoint balls in the plane. Leaving all but two strands unchanged in the diagrams of $L_1$ and $L_2$, we can obtain a diagram of a link of the form $L_1\#T_{(2,t)}\#L_2 =:L$, where $T_{(2,t)}$ denotes the $(2, t)$-torus link. If the initial diagrams $D_1,D_2$ of $L_1,L_2$ are (possibly loose) fishnets, it is easy to obtain a (necessarily loose) fishnet diagram for $L$: simply juxtapose $D_1$ and $D_2$ and insert a new column between them, in which one parameter is equal to $t$ and the remaining are equal to zero. (When $L_1$ and $L_2$  have different height, rows of zeroes are added to either $D_1$ or $D_2$.) In particular, if $L_1$ and $L_2$ admit quotients onto the Coxeter groups $G_1$ and $G_2$, then $L$ admits a quotient onto a Coxeter group determined by $G_1$, $G_2$, and $t$ in the natural way.
\end{enumerate}

Using the above remarks, one readily constructs new families of links which admit maximal rank Coxeter quotients. In particular, one can build regular fishnets of arbitrary dimensions which admit ``nonobvious" Coxeter quotients and satisfy MRC.

\subsubsection{Combining Fishnets} \label{sec:interweaving}

In Figure~\ref{fig:interweaving} we introduce the key construction used in Section~\ref{sec:sublink-thm-proof}. 

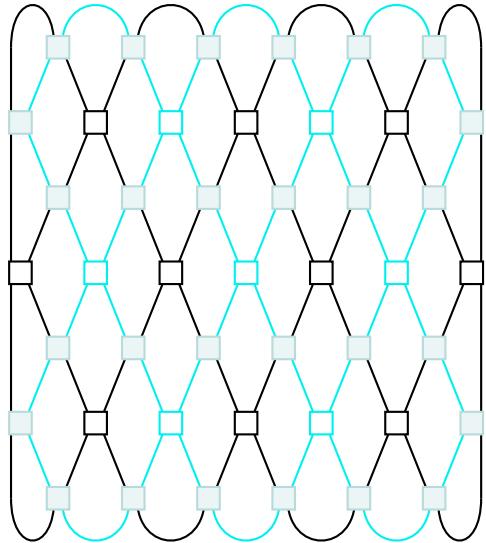
\begin{figure}[ht]
    \centering

    \begin{tikzpicture}
        [x=1mm,y=1mm,
        box1/.style={rectangle, inner sep=0mm, draw=blue1, fill=white, thick, minimum width=3mm, minimum height=3mm},
        box2/.style={rectangle, inner sep=0mm, draw=blue2, fill=white, thick, minimum width=3mm, minimum height=3mm},
        box3/.style={rectangle, inner sep=0mm, draw=blue3!55, fill=blue3!15, thick, minimum width=3mm, minimum height=3mm}]
        \colorlet{blue1}{black}
        \definecolor{blue2}{HTML}{00F0F0}
        \colorlet{blue3}{blue1!50!blue2!50}
        
        \draw[blue1, thick]
        (0.5,0) -- (4.5,10)
        (9.5,0) -- (5.5,10)
        (20.5,0) -- (24.5,10)
        (29.5,0) -- (25.5,10)
        (40.5,0) -- (44.5,10)
        (49.5,0) -- (45.5,10)
        ;
        \draw[blue1, thick]
        (4.5,10) -- (0.5,20)
        (5.5,10) -- (9.5,20)
        (24.5,10) -- (20.5,20)
        (25.5,10) -- (29.5,20)
        (44.5,10) -- (40.5,20)
        (45.5,10) -- (49.5,20)
        ;
        \draw[blue1, thick]
        (-0.5,20) -- (-4.5,30)
        (10.5,20) -- (14.5,30)
        (19.5,20) -- (15.5,30)
        (30.5,20) -- (34.5,30)
        (39.5,20) -- (35.5,30)
        (50.5,20) -- (54.5,30)
        ;
        \draw[blue1, thick]
        (-4.5,30) -- (-0.5,40)
        (14.5,30) -- (10.5,40)
        (15.5,30) -- (19.5,40)
        (34.5,30) -- (30.5,40)
        (35.5,30) -- (39.5,40)
        (54.5,30) -- (50.5,40)
        ;
        \draw[blue1, thick]
        (0.5,40) -- (4.5,50)
        (9.5,40) -- (5.5,50)
        (20.5,40) -- (24.5,50)
        (29.5,40) -- (25.5,50)
        (40.5,40) -- (44.5,50)
        (49.5,40) -- (45.5,50)
        ;
        \draw[blue1, thick]
        (4.5,50) -- (0.5,60)
        (5.5,50) -- (9.5,60)
        (24.5,50) -- (20.5,60)
        (25.5,50) -- (29.5,60)
        (44.5,50) -- (40.5,60)
        (45.5,50) -- (49.5,60)
        ;
        \draw[blue1, thick]
        (-6.25,0) -- (-6.25,60)
        ;
        \draw[blue1, thick]
        (56.25,0) -- (56.25,60)
        ;
        \draw[blue1, thick]
        (-6.25,0) .. controls (-6.25,-7.5) and (-0.5,-7.5) .. (-0.5,0)
        (10.5,0) .. controls (10.5,-7.5) and (19.5,-7.5) .. (19.5,0)
        (30.5,0) .. controls (30.5,-7.5) and (39.5,-7.5) .. (39.5,0)
        (50.5,0) .. controls (50.5,-7.5) and (56.25, -7.5) .. (56.25,0)
        ;
        \draw[blue1, thick]
        (-6.25,60) .. controls (-6.25,67.5) and (-0.5,67.5) .. (-0.5,60)
        (10.5,60) .. controls (10.5,67.5) and (19.5,67.5) .. (19.5,60)
        (30.5,60) .. controls (30.5,67.5) and (39.5,67.5) .. (39.5,60)
        (50.5,60) .. controls (50.5,67.5) and (56.25, 67.5) .. (56.25,60)
        ;
        \draw[blue2, thick]
        (-0.5,0) -- (-4.5,10)
        (10.5,0) -- (14.5,10)
        (19.5,0) -- (15.5,10)
        (30.5,0) -- (34.5,10)
        (39.5,0) -- (35.5,10)
        (50.5,0) -- (54.5,10)
        ;
        \draw[blue2, thick]
        (-4.5,10) -- (-0.5,20)
        (14.5,10) -- (10.5,20)
        (15.5,10) -- (19.5,20)
        (34.5,10) -- (30.5,20)
        (35.5,10) -- (39.5,20)
        (54.5,10) -- (50.5,20)
        ;
        \draw[blue2, thick]
        (0.5,20) -- (4.5,30)
        (9.5,20) -- (5.5,30)
        (20.5,20) -- (24.5,30)
        (29.5,20) -- (25.5,30)
        (40.5,20) -- (44.5,30)
        (49.5,20) -- (45.5,30)
        ;
        \draw[blue2, thick]
        (4.5,30) -- (0.5,40)
        (5.5,30) -- (9.5,40)
        (24.5,30) -- (20.5,40)
        (25.5,30) -- (29.5,40)
        (44.5,30) -- (40.5,40)
        (45.5,30) -- (49.5,40)
        ;
        \draw[blue2, thick]
        (-0.5,40) -- (-4.5,50)
        (10.5,40) -- (14.5,50)
        (19.5,40) -- (15.5,50)
        (30.5,40) -- (34.5,50)
        (39.5,40) -- (35.5,50)
        (50.5,40) -- (54.5,50)
        ;
        \draw[blue2, thick]
        (-4.5,50) -- (-0.5,60)
        (14.5,50) -- (10.5,60)
        (15.5,50) -- (19.5,60)
        (34.5,50) -- (30.5,60)
        (35.5,50) -- (39.5,60)
        (54.5,50) -- (50.5,60)
        ;
        \draw[blue2, thick]
        ;
        \draw[blue2, thick]
        ;
        \draw[blue2, thick]
        (0.5,0) .. controls (0.5,-7.5) and (9.5,-7.5) .. (9.5,0)
        (20.5,0) .. controls (20.5,-7.5) and (29.5,-7.5) .. (29.5,0)
        (40.5,0) .. controls (40.5,-7.5) and (49.5,-7.5) .. (49.5,0)
        ;
        \draw[blue2, thick]
        (0.5,60) .. controls (0.5,67.5) and (9.5,67.5) .. (9.5,60)
        (20.5,60) .. controls (20.5,67.5) and (29.5,67.5) .. (29.5,60)
        (40.5,60) .. controls (40.5,67.5) and (49.5,67.5) .. (49.5,60)
        ;
        \node[box3](72) at (0,0){};
        \node[box3](74) at (10,0){};
        \node[box3](76) at (20,0){};
        \node[box3](78) at (30,0){};
        \node[box3](710) at (40,0){};
        \node[box3](712) at (50,0){};

        \node[box3](61) at (-5,10){};
        \node[box1](63) at (5,10){};
        \node[box2](65) at (15,10){};
        \node[box1](67) at (25,10){};
        \node[box2](69) at (35,10){};
        \node[box1](611) at (45,10){};
        \node[box3](613) at (55,10){};

        \node[box3](52) at (0,20){};
        \node[box3](54) at (10,20){};
        \node[box3](56) at (20,20){};
        \node[box3](58) at (30,20){};
        \node[box3](510) at (40,20){};
        \node[box3](512) at (50,20){};

        \node[box1](41) at (-5,30){};
        \node[box2](43) at (5,30){};
        \node[box1](45) at (15,30){};
        \node[box2](47) at (25,30){};
        \node[box1](49) at (35,30){};
        \node[box2](411) at (45,30){};
        \node[box1](413) at (55,30){};

        \node[box3](32) at (0,40){};
        \node[box3](34) at (10,40){};
        \node[box3](36) at (20,40){};
        \node[box3](38) at (30,40){};
        \node[box3](310) at (40,40){};
        \node[box3](312) at (50,40){};

        \node[box3](21) at (-5,50){};
        \node[box1](23) at (5,50){};
        \node[box2](25) at (15,50){};
        \node[box1](27) at (25,50){};
        \node[box2](29) at (35,50){};
        \node[box1](211) at (45,50){};
        \node[box3](213) at (55,50){};

        \node[box3](12) at (0,60){};
        \node[box3](14) at (10,60){};
        \node[box3](16) at (20,60){};
        \node[box3](18) at (30,60){};
        \node[box3](110) at (40,60){};
        \node[box3](112) at (50,60){};        
    \end{tikzpicture}
    
    \caption{Interweaving fishnets: a fishnet $L$ made up of black sublink $L_1$ and blue sublink $L_2$, with gray shaded boxes only determined mod~$2$. If $L_1$ and $L_2$ each admit maximal rank quotients (determined by the unshaded boxes), so does $L$. Independently, if $L$ satisfies the gcd condition Proposition~\ref{thm:fishnet-quotients} (determined by a subset of the shaded boxes), then $L$ admits a different maximal rank quotient.
    }
    \label{fig:interweaving}
\end{figure}
We start with two fishnets, $L_1$ and $L_2$, in bridge position with $b$ and $b-1$ bridges, respectively; they may have any number of rows\footnote{This may not be immediately apparent from Figure~\ref{fig:interweaving}. If one of the sublinks, $L_i$, has fewer rows than the other, we add rows of boxes labeled ``0" to $L_i$ until the heights of the two links match. Then $L_1$ and $L_2$ can be interweaved. A similar idea is presented in Figure~\ref{fig:almost-mrc}, in which the process described for $L_i$ is applied to the blue sublink.}. The two links are depicted, superimposed, in Figure~\ref{fig:interweaving}. $L_1$ is black, $L_2$ is  blue. When the diagrams are simply placed on top of each other, the gray boxes contain single crossings. We use this setup to construct new families of fishnets in $(2b-1)$-bridge position by changing the values of the parameters in the gray boxes. We study the existence of Coxeter quotients of Coxeter rank $2b-1$ for these fishnets.

As in observation~(\ref{obv: direct sum}) above, if sublinks $L_1$ and $L_2$ admit Coxeter quotients of rank $b$ and $b-1$, then $L$ admits a maximal rank Coxeter quotient given by the direct sum of these two. Unique to this arrangement is that $L$ may be constructed so that it admits a maximal rank quotient {\it regardless} of whether sublinks $L_1$ and $L_2$ do. Indeed, the gcd condition of Proposition~\ref{thm:fishnet-quotients} --- namely $d_j>1$ for all even $j$ --- depends only on parameters of $L$ which are in gray boxes in Figure~\ref{fig:interweaving}. These parameters are fixed mod~2 in order to maintain the sublink decomposition of $L$ into $L_1$ and $L_2$, but can always be chosen to satisfy this gcd condition in a way that preserves parity. Therefore, there is always a choice in the construction of $L$ which ensures $L_1 \cup L_2$ admits a maximal rank Coxeter quotient.

\section{Proof of Theorem~\ref{thm:sublink}} \label{sec:sublink-thm-proof}
This section presents the proof of Theorem~\ref{thm:sublink}. In Section~\ref{sec:fishnet-properties} we give some preliminaries; 
and in Section~\ref{sec:augmenting-L} we give the full construction.

\subsection{Component Number of Fishnets} \label{sec:fishnet-properties}

We begin with a discussion of the component number of a fishnet. This will be used to demonstrate that the links which satisfy the hypotheses of Proposition~\ref{thm:fishnet-quotients} can have any number of components between 1 and the bridge number.

Let $L_\mathfrak{t}$ denote a loose fishnet which is the plat closure of an $m$-stranded braid $\zeta_\mathfrak{t}$. Consider the homomorphism $\phi: B_m \twoheadrightarrow S_m$ from the braid group to the symmetric group given by $\phi: \sigma_i \mapsto (i~i{+}1)$. Using this map, we associate a permutation $\phi(\zeta_\mathfrak{t})$ to $L_\mathfrak{t}$ which can be used to obtain the component number of $L_\mathfrak{t}$. 

We will show that any component number between 1 and $\frac{m}{2}$ can be achieved by a fishnet of width $m$ satisfying a prescribed gcd condition on the parameters in even columns. Precisely, fix any even integer $m\geq 2$ and any set of odd integers $\delta_2, \delta_4 \dots, \delta_m$.  
The following lemma shows that there exists a transitive permutation in $S_m$ which is in the image of a fishnet $L_\mathfrak{t}$, whose parameter set $\mathfrak{t}$ has the property that $d_{2j}$ is divisible by $\delta_{2j}$ for all $j=1, 2,\dots, \frac{m}{2}$.

\begin{lemma}~\cite{Pfaff2024} \label{lem:comp-number}
    Fix $m \in \mathbb N$, and fix set $\Delta = \{\delta_2,\delta_4, \dots, \delta_{m{-}2}\}$ such that $\delta_{2j} \geq 3$ and $\delta_{2j} \equiv 1 \mod~2$ for all $\delta_{2j} \in \Delta$. Let $\mathcal L_\Delta$ be the set of strong fishnets $L_\mathfrak{t}$ defined by
    \[
    \mathcal L_\Delta := \left\{L_{\mathfrak{t}=(t_{1,2},t_{1,4}, \dots, t_{1,m{-}2}; \dots, t_{n,m-2})}\ \big\lvert\ |t_{i,j}|\geq 3 \ \forall i,j \text{ and $\delta_{2j}\mid t_{i,2j} \ \forall i \in \{1,3,\dots,n\}$}, n \in \mathbb N \right\}.
    \]
    Then for any $k \in \{1,2, \dots, \frac{m}{2}\}$, there exists a strong fishnet $L_{\mathfrak{t}_k} \in \mathcal L_\Delta$ which has $k$ components.
\end{lemma}

\begin{proof}  
    For $L_\mathfrak{t} \in \mathcal L_\Delta$, denote by $\zeta_\mathfrak{t}$ its corresponding braid, given explicitly in Definition~\ref{def:fishnet}. Let $\zeta_\Delta := \{\zeta_\mathfrak{t} \mid L_\mathfrak{t}\in\mathcal L_\Delta\} \subset B_m$. Since the component number of $L_\mathfrak{t}$ is determined by its image under map $\phi: B_m \twoheadrightarrow S_m$ where $\phi(\sigma_i) = (i\ i{+}1)$ as above, it suffices to show that the image of $\zeta_\Delta$ under $\phi$ surjects onto $S_m$. To do this, we will find fishnets in $\mathcal L_\Delta$ which map to a generating set for $S_m$, and then realize multiplication in $S_m$ by combining fishnets in a concrete way, defined below.

    We will define $L_{\mathfrak{t}_1}, L_{\mathfrak{t}_2} \in \mathcal L_\Delta$ which correspond to an $n$-cycle and a transposition in $S_m$, respectively\footnote{One can as well show that every transposition $(i~i{+}1)$ in $S_m$ is in the image of $\mathcal{L}_\Delta$ under $\phi$. If $i$ is odd, it is a straightforward generalization of $\mathfrak{t}_2$ to define a fishnet with three rows mapping to $(i~i{+}1)$.}. Let $\mathfrak{t}_1$ be a parameter vector which defines a particular fishnet $L_{\mathfrak{t}_1}$ of width $m$ and height $m{+}1$ with odd entries along a diagonal and even entries elsewhere. 
    Specifically, we set:
    \begin{align*}
        \text{for } i=j+1,~~t_{i,j} =\begin{cases}
            \delta_j, & i \text{ odd} \\
            3, & i \text{ even}
        \end{cases};
        \text{ and for } i\neq j+1,~~t_{i,j} =\begin{cases}
            2\delta_j, & i \text{ odd} \\
            4, & i \text{ even}
        \end{cases}.
    \end{align*}
    Let $\mathfrak{t}_2 = (\delta_2, 2\delta_4, \dots, 2\delta_{m{-}2})$, defining a fishnet with a single row. The parameter vectors $\mathfrak{t}_1$ and $\mathfrak{t}_2$ satisfy the necessary conditions to ensure that $L_{\mathfrak{t}_1},L_{\mathfrak{t}_2} \in \mathcal L_\Delta$. Hence, $\zeta_{\mathfrak{t}_1}, \zeta_{\mathfrak{t}_2} \in \zeta_\Delta$. Additionally, $\phi(\zeta_{\mathfrak{t}_1}) = (m\ m{-}1\ \dots\ 1)$ and $\phi(\zeta_{\mathfrak{t}_2}) = (2\ 3)$, so $\phi(\zeta_{\mathfrak{t}_1})$ and $\phi(\zeta_{\mathfrak{t}_2})$ generate $S_m$.

    We now describe how to realize multiplication in $S_m$ by ``composing" fishnet diagrams. It is not possible to directly stack one fishnet of width $m$ on top of another due to the parity constraint on rows. Instead, we do the following. Let $\mathfrak{r} = (r_{1,2},r_{1,4}, \dots, r_{1,m{-}2}; \dots, r_{n,m-2})$ and $\mathfrak{s} = (s_{1,2},s_{1,4}, \dots, s_{1,m{-}2}; \dots, s_{n,m-2})$ be parameter vectors which determine fishnets $L_\mathfrak{r}$ and $L_\mathfrak{s}$ of width $m$. Their associated braids are denoted $\zeta_\mathfrak{r}, \zeta_\mathfrak{s}$. Then  $\phi(\zeta_\mathfrak{r})\phi(\zeta_\mathfrak{s})=\phi(\zeta_\mathfrak{r} \ast \zeta_\mathfrak{s})$ in $S_m$, where $\ast$ denotes composition in $B_m$.
    In order to compose the fishnets $L_\mathfrak{r}$ and $L_\mathfrak{s}$, we will stack the fishnet diagram of $L_\mathfrak{r}$ above the fishnet diagram of $L_\mathfrak{s}$ with an additional row of even parameters between them, resulting in the fishnet $L_{\mathfrak{r}\ast \mathfrak{s}}$, where $$\mathfrak{r}\ast \mathfrak{s} := (s_{1,2},s_{1,4}, \dots, s_{1,m{-}2}; \dots, s_{n,m-2};4,4,\dots,4;r_{1,2},r_{1,4}, \dots, r_{1,m{-}2}; \dots, r_{n,m-2}).$$
    Then $\phi(\zeta_{\mathfrak{r} \ast \mathfrak{s}}) = \phi(\zeta_\mathfrak{r})(1)\phi(\zeta_\mathfrak{s}) = \phi(\zeta_\mathfrak{r})\phi(\zeta_\mathfrak{s}) \in S_m$. By construction, all integer parameters in $L_{\mathfrak{r}\ast \mathfrak{s}}$ are at least 3, so $L_{\mathfrak{r}\ast \mathfrak{s}}$ is a strong fishnet, as desired. 

    Since there is a composition of fishnets associated to multiplication in $S_m$, we conclude from the surjectivity of $\phi$ that for every permutation in $S_m$, there exists a fishnet diagram in $\mathcal L_\Delta$ whose underlying braid word is mapped to this permutation by $\phi$. It is clear that the component number of a fishnet $L_\mathfrak{t}$ is determined by $\phi(\zeta_\mathfrak{t})$ and that, if every permutation of the strands can be achieved, then so can every component number. Since the image of $\zeta_\Delta$ under $\phi$ surjects onto $S_m$, we conclude that fishnets in $\mathcal L_\Delta$ achieve every  component number $k \in \{1,2,\dots,\frac{m}{2}\}$. 
\end{proof}
Note that every fishnet in the set $\mathcal L_\Delta$ satisfies the gcd condition in Proposition~\ref{thm:fishnet-quotients}, and so admits a maximal rank Coxeter quotient. A quotient admitted by all fishnets in $\mathcal L_\Delta$ is the Coxeter group given by
\[
G_\Delta =\left\langle a_1, a_2, \dots, a_{\frac{m}{2}}\ \big\lvert \ a_1^2=a_2^2\dots = a_{\frac{m}{2}}^2=1 \text{ and } (a_ja_{j+1})^{\delta_{2j}}=1 \text{ for all } j \in\{1,...,\tfrac{m}{2}{-}1\} \right\rangle.
\]
The group $G_\Delta$ is a quotient (possibly trivially so) of each of the Coxeter groups $G_\mathfrak{t}$, for $\mathfrak{t}$ such that $L_\mathfrak{t}\in\mathcal{L}_\Delta.$ (Here, $G_\mathfrak{t}$ is as defined in the proof of Proposition~\ref{thm:fishnet-quotients} in terms of gcds of integers in $\mathfrak{t}$.) By application of Lemma~\ref{lem:comp-number}, the set of fishnets of width $m$ which admit $G_\Delta$ as a quotient contains fishnets of every component number from $1$ to $\frac{m}{2}$.

\subsection{Proof of Theorem~\ref{thm:sublink}} \label{sec:augmenting-L}

Let $L$ be a link in bridge position with $b$ bridges, presented in a diagram $D$ with $b$ local maxima with respect to the $y$-axis in the plane of projection. Note that we do not require $b=\beta(L)$ and indeed our argument applies to any link in any bridge position, except the case where $L$ is an unknot and $b=\beta(L)=1$. 

Our proof can be split into two steps. First, we use $D$ to construct a new diagram, $D'$, of $L$ such that $D'$ is a loose fishnet of width $2b$. We note that in most cases the process of obtaining $D'$ increases the crossing number. The second step is to interweave $U,$ an unknot in bridge position with $b-1$ bridges, through $D'$ to obtain the desired link $L\cup U.$

{\it Step one.} A sequence of planar isotopies and Reidemeister II moves (potentially increasing the crossing number) transforms $D$ into a diagram $D_1$ which is the plat closure of a braid. Moreover, we can assume that all local maxima of $D_1$ are at the same height, 
say $y=a$, and similarly, that all local minima are at the same height, $y=b$. We require that all crossings be at different heights between $a$ and $b$, so that we can read off a braid word $\zeta^\ast=\sigma_{i_1}^{r_1}\sigma_{i_2}^{r_2}\dots \sigma_{i_k}^{r_k}$ from $D_1$. 
Here each $i_j\in\{1, 2, \dots, 2b{-}1\}$ and $\sigma_{i_j}$ denotes the corresponding generator of the braid group. Of course, $\zeta^\ast$ is not uniquely determined by $D$. We assume that $\zeta^\ast$ is reduced in the sense that $i_j\neq i_{j+1}$ for all $j.$

In our construction, for each $j=1, 2, \dots, k{-}1,$  the term $\sigma_{i_j}^{r_j}$ in $\zeta^\ast$ will take up two rows in a loose fishnet diagram of $L$. The last term, $\sigma_{i_k}^{r_k}$, will take up a single row. As a result, there will be $2k-1$ rows in the diagram constructed, of which $k-1$ rows will consist entirely of zeroes. Clearly, when $k>1$ the fishnet produced in this manner will not have minimal height. This inefficiency does not increase the bridge number of $L\cup  U$, so we do not strive for minimality.

Consider $\sigma_{i_1}^{r_1}$. It will contribute a box labeled $r_1$ in column $i_1$. This box will be either in row $1$ or in row $2,$ depending on the parity of $i_1$. The remaining parameters in the first two rows will be zero. We proceed in the same manner for all $j<k$. That is, $\sigma_{i_j}^{r_j}$ determines rows $2j-1$ and $2j$. Again, all entries in these rows are zero, except for the entry in column $i_j,$ which is equal to $r_j$. 

Without loss of generality, we may assume that $i_k$ is even. (If it is odd, then the plat closures of $\sigma_{i_1}^{r_1}\sigma_{i_2}^{r_2}\dots \sigma_{i_k}^{r_k}$ and $\sigma_{i_1}^{r_1}\sigma_{i_2}^{r_2}\dots \sigma_{i_{k-1}}^{r_{k{-}1}}$ represent the same link.) The last, $(2k-1)$-th, row of $D'$ will then consist of $r_k$ in the $i_k$-th column and zeroes in the other columns. We do not follow this by a row of zeroes. This completes the construction of a loose fishnet diagram $D'$ for $L.$

{\it Step two.} Having presented $L$ as a loose fishnet, we now let the diagram $D'$ play the role of the black sublink in Figure~\ref{fig:almost-mrc}. (This is the step where our construction uses the assumption $b>1.$) One possible way of choosing the desired unknot $U\subset S^3\backslash L$ is pictured in blue. In this construction, the link $L\cup U$ satisfies the hypotheses of Proposition~\ref{thm:fishnet-quotients} with $d_{2j}=3$ for all $j\in \{1, 2,\dots \frac{m}{2}{-}1\}$.\footnote{It is evident in the figure that, for any choice of odd integers $\delta_{2j}$, we can just as easily embed $U$ in $S^3\backslash L$ so that $\delta_{2j}\mid d_{2j}$. This allows us to construct different quotients, depending on the choice of $U$.} Therefore, $L\cup U$ admits as a quotient the Coxeter group 
\[
G_{2b}:=\left\langle a_1, a_2, \dots, a_{\frac{m}{2}}\ \big\lvert \ a_1^2=a_2^2\dots = a_{\frac{m}{2}}^2=1 \text{ and } (a_ja_{j+1})^{3}=1 \text{ for all } j \in\{1,...,\tfrac{m}{2}{-}1\} \right\rangle.
\]
This suffices to establish the equality 
\begin{equation} \label{eq:almost-mrc}
\beta(L\cup U)=\mu(L\cup U)= \frac{m}{2}=2b-1.    
\end{equation}
To determine the {\it rank} of $\pi_1(S^3\backslash(L\cup U)),$ we consider a further quotient of $G_{2b}$ (and hence of the group of $L\cup U$), defined as follows:
\[
G_{2b}':=\left\langle a_1, a_2, \dots, a_{\frac{m}{2}}\ \big\lvert \ a_1^2=a_2^2\dots = a_{\frac{m}{2}}^2=1 \text{ and } (a_ia_{j})^{3}=1 \text{ for all } i\neq j \right\rangle.
\]
The proof of~\cite[Lemma~3.2(c)]{kaufmann1992rank} applies without modification to show that $\text{rank}(G_{2b}')=\frac{m}{2}$. This then implies that $\text{rank}(\pi_1(S^3\backslash(L\cup U)))=\mu(L\cup U) = \frac{m}{2},$ concluding the proof.\qed

\begin{remark} 
By letting all twist regions between $L$ and $U$ contain $\pm3$ crossings, we have arranged that $G_{2b}$ admits the symmetric group $S_{2b}$ as a further quotient. An explicit quotient 
\[ 
\pi_1(S^3\backslash(L\cup U))\twoheadrightarrow S_{2b},
\]
mapping meridians to transpositions, is given in Figure~\ref{fig:almost-mrc}. This quotient suffices to establish MRC for $L\cup U$ without appealing Proposition~\ref{thm:fishnet-quotients} or the more general theory of Coxeter groups.
\end{remark}

\begin{remark}
Given a link $L\subset S^3$, let $E(L)$ denote its exterior. Recall that the tunnel number of $L$, denoted $t(L)$, is the minimal number of arcs properly embedded in $E(L)$ such that the exterior of the arcs in $E(L)$ is a handlebody. By definition, $t(L)=g(E(L))-1$, where $g(E(L))$ is the Heegaard genus of $E(L)$. One can show $g(E(L))\geq {\rm rank}(\pi_1(S^3\backslash L))$ by an application of the Seifert-Van Kampen theorem to any minimal Heegaard splitting of $E(L)$. Additionally, we have that $\mu(L)\geq {\rm rank}(\pi_1(S^3\backslash L))$ by the definition of meridional rank, and that $\beta(L)\geq g(E(L))$ since any $b$-bridge splitting of $L$ gives a genus $b$ Heegaard splitting for $E(L)$. 
Putting these together, we have $\beta(L)\geq g(E(L))= t(L)+1\geq {\rm rank}(\pi_1(S^3\backslash L)).$ 
Therefore, the links $L\cup U$ obtained by our construction have the property that
\[\beta(L\cup U)=\mu(L\cup U)= t(L\cup U)+1 ={\rm rank}(\pi_1(S^3\backslash(L\cup U)))=2b-1.\]
\end{remark}

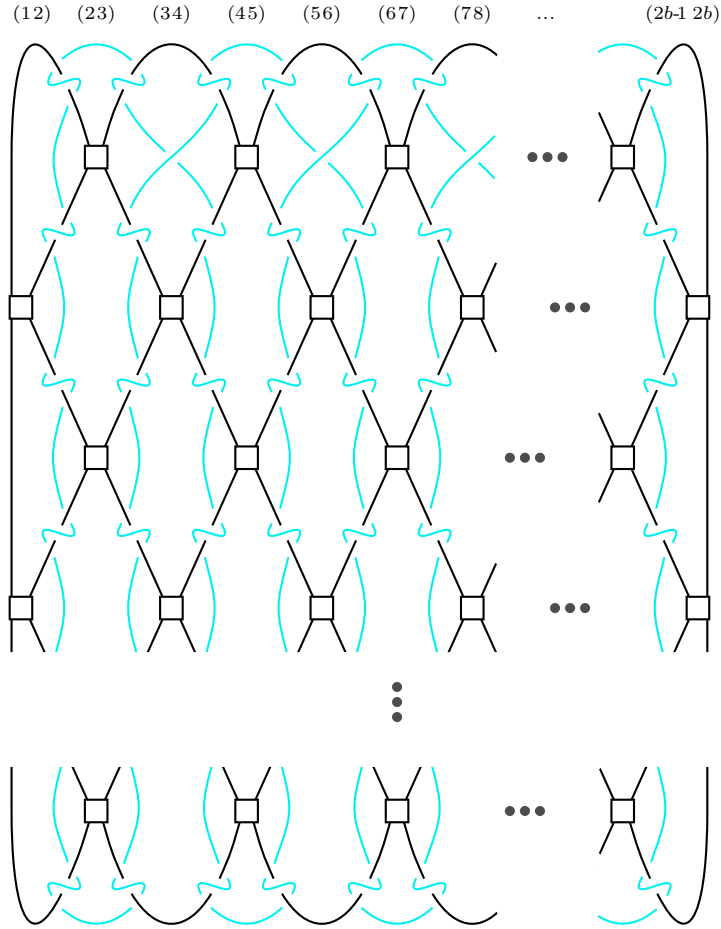
\begin{figure}[ht]
    \centering
    \begin{tikzpicture}
        [x=1mm,y=1mm,
        every path/.style = {thick},
        box/.style={rectangle, inner sep=0mm, draw=black, fill=white, thick, minimum width=3mm, minimum height=3mm}]
        \definecolor{unknot}{HTML}{00F0F0}

        \draw (-1.25,10) -- (-1.25,90);
        \draw (81.25,10) -- (81.25,90);
        \begin{scope}[unknot]
            \foreach \x in {10,30,50,70}{
            \foreach \y in {0}{
                \draw (\x,\y) +(-5,0) .. controls +(2.05,-6.65) and +(-2.05,-6.65) .. ++(5,0);
            }
            \foreach \y in {10}{
                \draw (\x,\y) +(-5,10) .. controls +(-2,-7.5) and +(-2,7.5) .. ++(-5,-10);
                \draw (\x,\y) +(5,10) .. controls +(2,-7.5) and +(2,7.5) .. ++(5,-10);
            }
            \foreach \y in {30}{
                \draw (\x,\y) +(-5.75,7.5) .. controls +(2,-7.5) and +(2,7.5) .. ++(-5.75,-7.5);
                \draw (\x,\y) +(5.75,7.5) .. controls +(-2,-7.5) and +(-2,7.5) .. ++(5.75,-7.5);
            }
            \foreach \y in {70}{
                \draw (\x,\y) +(-5.75,7.5) .. controls +(2,-7.5) and +(2,7.5) .. ++(-5.75,-7.5);
                \draw (\x,\y) +(5.75,7.5) .. controls +(-2,-7.5) and +(-2,7.5) .. ++(5.75,-7.5);
            }
            \foreach \y in {100}{
                \draw (\x,\y) +(-4.8,2.8) .. controls +(3.2,2.92) and +(-3.2,2.92) .. ++(4.8,2.8);
            }
        }
        \foreach \x in {10,30,50}{
            \foreach \y in {50}{
                \draw (\x,\y) +(-4.25,7.5) .. controls +(-2,-7.5) and +(-2,7.5) .. ++(-4.25,-7.5);
                \draw (\x,\y) +(4.25,7.5) .. controls +(2,-7.5) and +(2,7.5) .. ++(4.25,-7.5);
            }
        }
        \foreach \y in {90}{
            \draw (10,\y) +(-3.5,7.5) .. controls +(-2.75,-7.5) and +(-2,7.5) .. ++(-4.25,-7.5);
            \draw (70,\y) +(5,10) .. controls +(2,-7.5) and +(2,7.5) .. ++(5,-10);
            \foreach \x in {20,40,60}{
                \draw (\x,\y) +(-6.85,8.5) .. controls +(4.1,-8.5) and +(-3,7.5) .. ++(5.75,-7.5);
                \draw[white, fill=white] (\x,\y) circle (1mm);
                \draw (\x,\y) +(6.85,8.5) .. controls +(-4.1,-8.5) and +(3,7.5) .. ++(-5.75,-7.5);
            }  
        }
        \end{scope}
        \foreach \x in {5,25,45}{
            \foreach \y in {40,80}{
                \draw[unknot] (\x,\y) +(1.5,1.5) .. controls +(3,-5) and +(-3,5) .. ++(-1.5,-1.5);
                \draw[white, fill=white] (\x,\y) +(1,2.2) circle (1mm);
                \draw[white, fill=white] (\x,\y) +(-1,-2.2) circle (1mm);
            }
            \foreach \y in {60}{
                \draw[unknot] (\x,\y) +(-1.5,1.5) .. controls +(-3,-5) and +(3,5) .. ++(1.5,-1.5);
                \draw[white, fill=white] (\x,\y) +(-1,2.2) circle (1mm);
                \draw[white, fill=white] (\x,\y) +(1,-2.2) circle (1mm);
            }
        }
        \foreach \x in {15,35,55}{
            \foreach \y in {40,80}{
                \draw[unknot] (\x,\y) +(-1.5,1.5) .. controls +(-3,-5) and +(3,5) .. ++(1.5,-1.5);
                \draw[white, fill=white] (\x,\y) +(-1,2.2) circle (1mm);
                \draw[white, fill=white] (\x,\y) +(1,-2.2) circle (1mm);
            }
            \foreach \y in {60}{
                \draw[unknot] (\x,\y) +(1.5,1.5) .. controls +(3,-5) and +(-3,5) .. ++(-1.5,-1.5);
                \draw[white, fill=white] (\x,\y) +(1,2.2) circle (1mm);
                \draw[white, fill=white] (\x,\y) +(-1,-2.2) circle (1mm);
            }
        }
        \foreach \x in {5.75}{
            \foreach \y in {100}{
                \draw[unknot] (\x,\y) +(-1.5,1.5) .. controls +(-3,-5) and +(3,5) .. ++(1.5,-1.5);
                \draw[white, fill=white] (\x,\y) +(-1,2.2) circle (1mm);
                \draw[white, fill=white] (\x,\y) +(1,-2.2) circle (1mm);
            }
        }
        \foreach \x in {25.7,45.7}{
            \foreach \y in {100}{
                \draw[unknot] (\x,\y) +(-1.5,1.5) .. controls +(-3,-5) and +(3,5) .. ++(1.5,-1.5);
                \draw[white, fill=white] (\x,\y) +(-1.15,2) circle (1mm);
                \draw[white, fill=white] (\x,\y) +(1.15,-2) circle (1mm);
            }
        }
        \begin{scope}[xshift=-0.7mm]
        \foreach \x in {15,35,55}{
            \foreach \y in {100}{
                \draw[unknot] (\x,\y) +(1.5,1.5) .. controls +(3,-5) and +(-3,5) .. ++(-1.5,-1.5);
                \draw[white, fill=white] (\x,\y) +(1.15,2) circle (1mm);
                \draw[white, fill=white] (\x,\y) +(-1.15,-2) circle (1mm);
            }
        }
        \end{scope}
        \foreach \y in {10}{
            \foreach \x in {0}{
                \draw (\x,\y) +(-1.25,0) .. controls +(0,-20) and +(-6,-20) .. ++(9.5,0);
            }
            \foreach \x in {10,30,50}{
                \draw (\x,\y) +(0.5,0) .. controls +(5,-20) and +(-5,-20) .. ++(19.5,0);
            }
            \foreach \x in {70}{
                \draw (\x,\y) +(0.5,0) .. controls +(6,-20) and +(0,-20) .. ++(11.25,0);
            }
            \foreach \x in {10,30,50,70}{
                \draw (\x,\y) +(0.5,0) -- ++(5,10);
                \draw (\x,\y) +(-0.5,0) -- ++(-5,10);
                \node[box] at (\x,\y){};
            }
        }
        \foreach \y in {30}{
            \foreach \x in {0,20,40,60}{
                \draw (\x,\y) +(0.5,0) -- ++(4.55,9);
                \draw (\x,\y) +(0.5,0) -- ++(5,-10);
                \draw (\x,\y) +(19.5,0) -- ++(15.45,9);
                \draw (\x,\y) +(19.5,0) -- ++(15,-10);
            }
            \foreach \x in {0,20,40,60,80}{
                \node[box] at (\x,\y){};
            }
        }
        \foreach \y in {50}{
            \foreach \x in {10,30,50,70}{
                \draw (\x,\y) +(0.5,0) -- ++(4.55,9);
                \draw (\x,\y) +(-0.5,0) -- ++(-4.55,9);
                \draw (\x,\y) +(0.5,0) -- ++(4.55,-9);
                \draw (\x,\y) +(-0.5,0) -- ++(-4.55,-9);
                \node[box] at (\x,\y){};
            }
        }
        \foreach \y in {70}{
            \foreach \x in {0,20,40,60}{
                \draw (\x,\y) +(0.5,0) -- ++(4.55,9);
                \draw (\x,\y) +(0.5,0) -- ++(4.55,-9);
                \draw (\x,\y) +(19.5,0) -- ++(15.45,9);
                \draw (\x,\y) +(19.5,0) -- ++(15.45,-9);
            }
            \foreach \x in {0,20,40,60,80}{
                \node[box] at (\x,\y){};
            }
        }
        \foreach \y in {90}{
            \foreach \x in {0}{
                \draw (\x,\y) +(-1.25,0) .. controls +(0,11.5) and +(-4.8,9.5) .. ++(5.4,11);
                \draw (\x,\y) +(6.35,9) .. controls +(1.5,-3.5) .. ++(9.5,0);
            }
            \foreach \x in {10,30,50}{
                \draw (\x,\y) +(3.6,9) .. controls +(-1.55,-3.5) .. ++(0.5,0);
                \draw (\x+20,\y) +(-3.6,9) .. controls +(1.55,-3.5) .. ++(-0.5,0);
                \draw (\x,\y) +(4.7,11) .. controls +(3.25,5.35) and +(-3.25,5.35) .. ++(20-4.7,11);
            }
            \foreach \x in {70}{
                \draw (\x,\y) +(0.5,0) .. controls +(6,20) and +(0,20) .. ++(11.25,0);
            }
            \foreach \x in {10,30,50,70}{
                \draw (\x,\y) +(0.5,0) -- ++(4.55,-9);
                \draw (\x,\y) +(-0.5,0) -- ++(-4.55,-9);
                \node[box] at (\x,\y){};
            }
        }
        \draw[white, fill=white] (63.5,-10) rectangle (100,110);

        \begin{scope}[xshift=10mm]
        \draw (81.25,10) -- (81.25,90);
        \begin{scope}[unknot]
            \foreach \x in {70}{
            \foreach \y in {0}{
                \draw (\x,\y) +(-5,0) .. controls +(2.05,-6.65) and +(-2.05,-6.65) .. ++(5,0);
            }
            \foreach \y in {10}{
                \draw (\x,\y) +(-5,10) .. controls +(-2,-7.5) and +(-2,7.5) .. ++(-5,-10);
                \draw (\x,\y) +(5,10) .. controls +(2,-7.5) and +(2,7.5) .. ++(5,-10);
            }
            \foreach \y in {30}{
                \draw (\x,\y) +(5.75,7.5) .. controls +(-2,-7.5) and +(-2,7.5) .. ++(5.75,-7.5);
            }
            \foreach \y in {50}{
                \draw (\x,\y) +(-5,10) .. controls +(-2,-7.5) and +(-2,7.5) .. ++(-5,-10);
                \draw (\x,\y) +(4.25,7.5) .. controls +(2,-7.5) and +(2,7.5) .. ++(4.25,-7.5);
            }
            \foreach \y in {70}{
                \draw (\x,\y) +(5.75,7.5) .. controls +(-2,-7.5) and +(-2,7.5) .. ++(5.75,-7.5);
            }
            \foreach \y in {100}{
                \draw (\x,\y) +(-4.8,2.8) .. controls +(3.2,2.92) and +(-3.2,2.92) .. ++(4.8,2.8);
            }
        }
        \foreach \y in {90}{
            \draw (70,\y) +(3.5,7.5) .. controls +(2.75,-7.5) and +(2,7.5) .. ++(4.25,-7.5);
        }
        \end{scope}
        \foreach \x in {75}{
            \foreach \y in {40,80}{
                \draw[unknot] (\x,\y) +(-1.5,1.5) .. controls +(-3,-5) and +(3,5) .. ++(1.5,-1.5);
                \draw[white, fill=white] (\x,\y) +(-1,2.2) circle (1mm);
                \draw[white, fill=white] (\x,\y) +(1,-2.2) circle (1mm);
            }
            \foreach \y in {60}{
                \draw[unknot] (\x,\y) +(1.5,1.5) .. controls +(3,-5) and +(-3,5) .. ++(-1.5,-1.5);
                \draw[white, fill=white] (\x,\y) +(1,2.2) circle (1mm);
                \draw[white, fill=white] (\x,\y) +(-1,-2.2) circle (1mm);
            }
        }
        \foreach \x in {74.25}{
            \foreach \y in {100}{
                \draw[unknot] (\x,\y) +(1.5,1.5) .. controls +(3,-5) and +(-3,5) .. ++(-1.5,-1.5);
                \draw[white, fill=white] (\x,\y) +(1,2.2) circle (1mm);
                \draw[white, fill=white] (\x,\y) +(-1,-2.2) circle (1mm);
            }
        }
        \foreach \y in {10}{
            \foreach \x in {70}{
                \draw (\x,\y) +(0.5,0) .. controls +(6,-20) and +(0,-20) .. ++(11.25,0);
                \draw (\x,\y) +(0.5,0) -- ++(4.55,9);
                \draw (\x,\y) +(-0.5,0) -- ++(-4.55,9);
                \draw (\x,\y) +(-0.5,0) -- ++(-4.55,-9);
                \node[box] at (\x,\y){};
            }
        }
        \foreach \y in {30}{
            \foreach \x in {60}{
                \draw (\x,\y) +(0.5,0) -- ++(4.55,9);
                \draw (\x,\y) +(0.5,0) -- ++(4.55,-9);
                \draw (\x,\y) +(19.5,0) -- ++(15.45,9);
                \draw (\x,\y) +(19.5,0) -- ++(15.45,-9);
            }
            \foreach \x in {60,80}{
                \node[box] at (\x,\y){};
            }
        }
        \foreach \y in {50}{
            \foreach \x in {70}{
                \draw (\x,\y) +(0.5,0) -- ++(4.55,9);
                \draw (\x,\y) +(-0.5,0) -- ++(-4.55,9);
                \draw (\x,\y) +(0.5,0) -- ++(4.55,-9);
                \draw (\x,\y) +(-0.5,0) -- ++(-4.55,-9);
                \node[box] at (\x,\y){};
            }
        }
        \foreach \y in {70}{
            \foreach \x in {60}{
                \draw (\x,\y) +(0.5,0) -- ++(4.55,9);
                \draw (\x,\y) +(0.5,0) -- ++(5,-9);
                \draw (\x,\y) +(19.5,0) -- ++(15.45,9);
                \draw (\x,\y) +(19.5,0) -- ++(15.45,-9);
            }
            \foreach \x in {60,80}{
                \node[box] at (\x,\y){};
            }
        }
        \foreach \y in {90}{
            \foreach \x in {70}{
                \draw (\x+10,\y) +(1.25,0) .. controls +(0,11.5) and +(4.8,9.5) .. ++(-5.4,11);
                \draw (\x+10,\y) +(-6.35,9) .. controls +(-1.5,-3.5) .. ++(-9.5,0);
                \draw (\x,\y) +(0.5,0) -- ++(4.55,-9);
                \draw (\x,\y) +(-0.5,0) -- ++(-4.55,-9);
                \draw (\x,\y) +(-0.5,0) -- ++(-4.55,9);
                \node[box] at (\x,\y){};
            }
        }
        \draw[white, fill=white] (57.5,-10) rectangle (66.5,110);
        \draw[white, fill=white] (-15,-10) rectangle (85,24);
        \end{scope}
        
        \begin{scope}[yshift=-7mm]
        \draw (-1.25,10) -- (-1.25,20);
        \draw (91.25,10) -- (91.25,20);
        \begin{scope}[unknot]
            \foreach \x in {10,30,50,80}{
            \foreach \y in {0}{
                \draw (\x,\y) +(-4.8,-2.8) .. controls +(3.2,-2.92) and +(-3.2,-2.92) .. ++(4.8,-2.8);
            }
            \foreach \y in {10}{
                \draw (\x,\y) +(-4.25,7.5) .. controls +(-2,-7.5) and +(-2.75,7.5) .. ++(-3.5,-7.5);
                \draw (\x,\y) +(4.25,7.5) .. controls +(2,-7.5) and +(2.75,7.5) .. ++(3.5,-7.5);
            }
        }
        \end{scope}
        \foreach \x in {5.75}{
            \foreach \y in {0}{
                \draw[unknot] (\x,\y) +(1.5,1.5) .. controls +(3,-5) and +(-3,5) .. ++(-1.5,-1.5);
                \draw[white, fill=white] (\x,\y) +(1,2.2) circle (1mm);
                \draw[white, fill=white] (\x,\y) +(-1,-2.2) circle (1mm);
            }
        }
        \foreach \x in {25.7,45.7}{
            \foreach \y in {0}{
                \draw[unknot] (\x,\y) +(1.5,1.5) .. controls +(3,-5) and +(-3,5) .. ++(-1.5,-1.5);
                \draw[white, fill=white] (\x,\y) +(1.15,2) circle (1mm);
                \draw[white, fill=white] (\x,\y) +(-1.15,-2) circle (1mm);
            }
        }
        \begin{scope}[xshift=-0.7mm]
        \foreach \x in {15,35,55}{
            \foreach \y in {0}{
                \draw[unknot] (\x,\y) +(-1.5,1.5) .. controls +(-3,-5) and +(3,5) .. ++(1.5,-1.5);
                \draw[white, fill=white] (\x,\y) +(-1.15,2) circle (1mm);
                \draw[white, fill=white] (\x,\y) +(1.15,-2) circle (1mm);
            }
        }
        \end{scope}
        \foreach \x in {84.25}{
            \foreach \y in {0}{
                \draw[unknot] (\x,\y) +(-1.5,1.5) .. controls +(-3,-5) and +(3,5) .. ++(1.5,-1.5);
                \draw[white, fill=white] (\x,\y) +(-1,2.2) circle (1mm);
                \draw[white, fill=white] (\x,\y) +(1,-2.2) circle (1mm);
            }
        }
        \foreach \y in {10}{
            \foreach \x in {0}{
                \draw (\x,\y) +(-1.25,0) .. controls +(0,-11.5) and +(-4.8,-9.5) .. ++(5.4,-11);
                \draw (\x,\y) +(6.35,-9) .. controls +(1.5,3.5) .. ++(9.5,0);
            }
            \foreach \x in {10,30,50}{
                \draw (\x,\y) +(3.6,-9) .. controls +(-1.55,3.5) .. ++(0.5,0);
                \draw (\x+20,\y) +(-3.6,-9) .. controls +(1.55,3.5) .. ++(-0.5,0);
                \draw (\x,\y) +(4.7,-11) .. controls +(3.25,-5.35) and +(-3.25,-5.35) .. ++(20-4.7,-11);
            }
            \foreach \x in {90}{
                \draw (\x,\y) +(1.25,0) .. controls +(0,-11.5) and +(4.8,-9.5) .. ++(-5.4,-11);
                \draw (\x,\y) +(-6.35,-9) .. controls +(-1.5,3.5) .. ++(-9.5,0);
            }
            \foreach \x in {10,30,50}{
                \draw (\x,\y) +(0.5,0) -- ++(5,10);
                \draw (\x,\y) +(-0.5,0) -- ++(-5,10);
                \node[box] at (\x,\y){};
            }
            \foreach \x in {80}{
                \draw (\x,\y) +(0.5,0) -- ++(5,10);
                \draw (\x,\y) +(-0.5,0) -- ++(-5,10);
                \draw (\x,\y) +(-0.5,0) -- ++(-5,-10);
                \node[box] at (\x,\y){};
            }
        }
        \draw[white, fill=white] (-5,16) rectangle (95,20);
        \draw[white, fill=white] (63.5,-5) rectangle (76.5,20);
        \end{scope}
        \foreach \i in {-1,0,1}{
            \draw[black!70, fill=black!70] (50,2*\i+17.5) circle (1.5pt);
            \draw[black!70, fill=black!70] (2*\i+70,90) circle (1.5pt);
            \draw[black!70, fill=black!70] (2*\i+73,70) circle (1.5pt);
            \draw[black!70, fill=black!70] (2*\i+67,50) circle (1.5pt);
            \draw[black!70, fill=black!70] (2*\i+73,30) circle (1.5pt);
            \draw[black!70, fill=black!70] (2*\i+67,3) circle (1.5pt);
        }
        \draw[white, fill=white]
        (64,-10) circle (.75mm)
        (63.6,36.7) circle (.75mm)
        (63.6,63.3) circle (.75mm)
        (63.6,76.7) circle (.75mm)
        (63.6,86.5) circle (.75mm)
        (63.6,93.5) circle (.75mm)
        (64,103) circle (.75mm)
        ;
        \draw[white, fill=white]
        (76,-10.4) circle (.75mm)
        (76.4,-3.4) circle (.75mm)
        (76.4,9.3) circle (.75mm)
        (76.4,43.3) circle (.75mm)
        (76.4,56.7) circle (.75mm)
        (76.4,83.3) circle (.75mm)
        (76.4,96.7) circle (.75mm)
        (76,103.4) circle (.75mm)
        ;
        \node at (1.5,109){$\scriptstyle(1\!\!\ 2)$};
        \node at (10,109){$\scriptstyle(2\!\!\ 3)$};
        \node at (20,109){$\scriptstyle(3\!\!\ 4)$};
        \node at (30,109){$\scriptstyle(4\!\!\ 5)$};
        \node at (40,109){$\scriptstyle(5\!\!\ 6)$};
        \node at (50,109){$\scriptstyle(6\!\!\ 7)$};
        \node at (60,109){$\scriptstyle(7\!\!\ 8)$};
        \node at (70,108.5){$\scriptstyle\dots$};
        \node at (88,109){$\scriptstyle(2b\text{-}\!1\!\ 2b)$};
        \end{tikzpicture}
    \caption{A link $L$ in fishnet position with $b$ bridges, colored black, and an unknot $U \subset S^3\backslash L$, colored blue, embedded in its exterior. The resulting link $L\cup U$ surjects onto the symmetric group $S_{2b},$ mapping meridians to transpositions, as shown. Thus, $\beta(L\cup U) = \mu(L \cup U)= 2b-1$.     }
    \label{fig:almost-mrc}
\end{figure}

\section{Distance} \label{sec:distance}

Suppose $M$ is a compact, orientable 3-manifold containing a properly embedded, compact, 1-manifold $\tau$. In what follows, we consider properly embedded, compact surfaces $F \subset M$ which are transverse to $\tau$. The points $\tau \cap F$ are the {\it punctures} on $F$. Two properly embedded punctured surfaces in $M$ are {\it equivalent} if they are properly isotopic via an isotopy that at every stage produces a surface transverse to $\tau$. We say two equivalent surfaces are transversely isotopic with respect to $\tau$ (or simply that they are isotopic, if both transversality and $\tau$ are understood from context). A curve $\gamma \subset F$ will be called {\it essential} if $\gamma$ is disjoint from $\tau$, is not boundary parallel, and doesn't bound a disk in $F$ with fewer than two punctures. The surface $F$ is {\it incompressible} in $M$ if there is no disk $D$ (called a {\it compressing disk}) in $M\backslash\tau$ 
such that $D \cap F=\partial D$ is an essential curve in $F$. If $F$ is a sphere, it is called an {\it inessential sphere} if it bounds a ball disjoint from $\tau$ or bounds a ball containing a single, boundary parallel subarc of $\tau$. A connected incompressible surface $F$ will be called {\it essential} if it is not an inessential sphere and if there is no parallelism transverse to $\tau$ 
between $F$ and some collection of components of $\partial M$. If $F$ is a closed, connected, punctured separating surface in $M$, then $F$ is {\it bicompressible} if there exists a compressing disk for $F$ contained to each of its sides.

 A {\it tangle} $(B, \tau)$ consists of a 3-ball $B$ and a properly embedded finite collection of arcs, $\tau$. A tangle $(B, \tau)$ is {\it irreducible} if every unpunctured 2-sphere in $(B, \tau)$ is inessential. The tangle $(B, \tau)$ is {\it trivial} if there is a collection of disjoint compressing disks for $\partial B$ contained in the complement of $\tau$ so that surgering $B$ along this collection results in components that are 3-balls, each containing a single boundary parallel arc.

Recall that an \emph{$n$-bridge sphere}, $\Sigma$, for a link $L$ in $S^3$ is a sphere meeting $L$ transversely in $2n$ points and dividing $(S^3,L)$ into two trivial tangles. 
The bridge spheres considered are not necessarily minimal. Note that when $\beta(L)=n,$ there is an $n$-bridge sphere $\Sigma$ for a representative of $L$, and we say that $\Sigma$ realizes the bridge number of $L$.

\subsection{Distance} Distance is a complexity measure for bridge spheres, obtained from the distance between disk sets in the curve complex.  Let $L \subset S^3$ be a link and let $\Sigma$ be any bridge sphere for $L$ separating $S^3$ into balls $B_1$ and $B_2$. Define the {\it curve complex} of $\Sigma$, denoted $\mathcal{C}(\Sigma)$, to be the graph whose vertices are isotopy classes of essential simple closed curves in $\Sigma$. Two vertices are connected by an edge if their corresponding curves can be isotoped in $\Sigma\backslash L$ to be disjoint. We endow the vertex set of the curve complex $\mathcal{C}(\Sigma)$ with a metric in the natural way: assign length one to each edge and let the distance between two vertices be the length of the shortest path between them.
Let the {\it disk set}, $\mathcal{D}_i \subset \mathcal{C}(\Sigma)$, be the set of vertices in $\mathcal{C}(\Sigma)$ which correspond to essential curves in $\Sigma$ that bound compressing disks in $B_i$, $i = 1,2$. We define the {\it distance of $\Sigma$}, denoted $d(\Sigma)$, to be
\[ d(\Sigma) = \min\{ d(c_1,c_2): c_i \in \mathcal{D}_i\}.\] 

If $\Sigma$ is punctured four or fewer times (i.e., when the curve complex is empty or disconnected), we define the distance of $\Sigma$ to be infinite. The {\it distance of $L$}, $d(L)$, is the maximum possible distance $d(\Sigma)$ of any bridge sphere $\Sigma$ for a representative of $L$ such that $\Sigma$ realizes the bridge number of $L$. It follows from \cite{To07} that as long as $\beta(L) \geq 3$, this maximum is a well-defined positive integer.

Given a bridge sphere $\Sigma$ for a link $L$, several authors have given upper bounds on the distance of $\Sigma$ in terms of the topology of certain essential \cite{BS05,BCJTT17} or bi-compressible \cite{To07,BCJTT17} surfaces in the link exterior. For simplicity, we have restricted ourselves to considering only knots and using only three such results when proving Theorem \ref{thm:low-d}. We remark that some of the bounds on distance in Theorem \ref{thm:low-d} can be improved by employing lengthier arguments or alternative upper bounds on distance. 

First, we need an upper bound on the distance of a knot in terms of the genus and number of boundary components of an essential (possibly meridional) surface in the knot exterior. The following theorem is a restatement of Theorem 5.7 of ~\cite{To07} for knots in $S^3$. This is very similar to the main theorem of \cite{BS05}, which uses a slightly different definition of distance. In the following theorem, $E(L)$ denotes the exterior of a link $L$ in $S^3$.

\begin{theorem}\label{MaggyThm}\cite{To07, BS05}
Suppose that $L$ is a non-trivial knot in $S^3$. Let $\Sigma$ be a bridge sphere for $L$. Suppose that $\overline{S} \subset S^3$ is a surface transverse to $L$ such that $S = \overline{S} \cap E(L)$ is not a sphere and is essential in $E(L)$. Then
\[
d(\Sigma) \leq \max\{3, 2g(S) + |\overline{S} \cap L|\}.\] 
 If, in addition, $\overline{S} = S$ and if $S$ is a torus, then $d(\Sigma) \leq 2$.
\end{theorem}

Next, we need an upper bound on the distance of a knot in terms of the genus and number of boundary components of an essential surface with non-meridional boundary components. The following theorem is a specialization of  Theorem 5.2 of \cite{BCJTT17} to knots in $S^3$.

\begin{theorem}\label{NonMeridionalBounds} \cite{BCJTT17}
Let $K$ be a nontrivial knot in $S^3$ with $\beta(K)\geq 3$.  Suppose that $F$ is an orientable essential surface properly embedded in the exterior of $K$ with non-empty, non-meridional boundary. Then

$$\beta(K)(d(K)-4)
\leq \frac{8g(F)-8}{|\partial F|}+4.$$
\end{theorem}

Finally, we need an upper bound on the distance of a knot in terms of the Heegaard genus of a manifold obtained by Dehn surgery. The following theorem is a specialization of Theorem 8.2 in \cite{BCJTT17} to knots in $S^3$.

\begin{theorem}\label{Bounding dist BJC}\cite{BCJTT17} Let $K$ be a non-trivial knot in $S^3$. Let $M'$ be the result of non-trivial Dehn surgery on $K$ and let $g$ be the Heegaard genus of $M'$.  Then
\[ d(K) \leq \max\Big\{\frac{2}{\beta(K)} + 4, \frac{4g}{\beta(K)} + 4, 2g+ 2\Big\}.\] 
\end{theorem}

\subsection{Distance and Fishnet Links}\label{DistanceandFish}

In \cite{johnson2016bridge}, Johnson and Moriah study the distance of strong fishnet links, described in the language of $2k$-plat projections. Given a diagram of a fishnet link as depicted in Figure \ref{Fig:JohnsonMoriah}, a horizontal line in the plane of projection that separates all maxima from all minima naturally corresponds to a bridge plane for the knot in $\mathbb{R}^3$, or, equivalently, a bridge sphere for the knot in $S^3$. We call any bridge sphere transversely isotopic to this the {\it induced bridge sphere} for a fishnet link. Denote by $\lceil x \rceil$ the ceiling function on $x$, equal to the smallest integer greater than or equal to $x$. Johnson and Moriah showed the following:

\begin{theorem}\label{thm:JohnsonMoriah}\cite{johnson2016bridge}
If $K \subset S^3$ is a strong fishnet link of width $m$ and height $n$, then $d(\Sigma) =  \lceil n /(m - 4)  \rceil$, where $\Sigma$ is the induced bridge sphere.  
\end{theorem}

\subsection{Distance and classes of links known to satisfy the MRC}\label{DistOfOtherMRCClasses}

In this section we show that knots for which the MRC has previously been established have either small bridge number or the property that the ratio of their bridge distance over their bridge number is less than 3. We have restricted ourselves to knots here. That said, the tools used would apply to links in all cases except generalized Montesinos links~\cite{LM93}, for which our proof would require a version of Theorem \ref{Bounding dist BJC} for links. That said, we believe that Theorem \ref{thm:low-d} holds for links as well.

\begin{theorem} \label{thm:low-d}
    Let $K$ be a knot shown to satisfy the MRC in one of ~\cite{boileau1985nombre, rost1987meridional, boileau1989orbifold, LM93, CH14, Corn14, boileau2017meridionalrank,
 baader2017symmetric, baader2019coxeter, baader2023bridge, Dutra22}. Then either $\beta(K)\leq 5$ or $d(K)\leq max(8, 2\beta(K)+2)$.
\end{theorem}

\begin{proof}
    In ~\cite{boileau1985nombre}, ~\cite{baader2019coxeter}, and ~\cite{baader2023bridge}, the authors prove the MRC for classes of Montesinos and arborescent knots. Every knot in these classes has the property that it is either bridge number at most 3 or that its exterior contains an essential meridional 4-punctured sphere. In the latter case, the distance of the knot is less than or equal to 4 by Theorem \ref{MaggyThm}.

    In ~\cite{rost1987meridional}, ~\cite{boileau1989orbifold}, ~\cite{Corn14}, ~\cite{CH14}, ~\cite{boileau2017meridionalrank}, ~\cite{baader2017symmetric}, and ~\cite{Dutra22}, the authors prove the MRC for classes of $3$-bridge knots, torus knots or satellite knots. Every knot in these classes has one or more of the following properties: its bridge number is at most 3, its exterior contains an essential non-meridional annulus, or its exterior contains an essential torus. In each of the latter two cases, the distance of the knot is less than or equal to three by Theorems \ref{NonMeridionalBounds} and \ref{MaggyThm}.

    In ~\cite{LM93}, the authors prove the MRC for a class of knots known as ``generalized Motesinos knots''. As a consequence of Theorem 0.1 in that same paper, every generalized Motesinos knot $K$ which the authors show satisfies the MRC additionally admits a non-trivial Dehn surgery resulting in a 3-manifold $M(K)$ such that the Heegaard genus of $M(K)$ is equal to $\beta(K)$. However, by Theorem \ref{Bounding dist BJC}, this non-trivial Dehn surgery implies that the distance of $K$ is at most the maximum of $8$ and $2\beta(K)+2$.

    A second class of knots which the authors of ~\cite{baader2019coxeter} prove satisfy the MRC are twisted links. Here we will give a brief overview of the construction of these links (restricting our attention to knots).  For additional details, we refer the reader to ~\cite{brunner1992geometric}, where twisted links were introduced; and to~\cite{baader2019coxeter}, where they were shown to satisfy MRC. 
    
    Let $D$ be a diagram of a knot $L$, and let $F$ be one of the two surfaces with boundary $L$ obtained from a checkerboard coloring of the regions in the plane determined by $D$. We regard the surface $F$ as a union of disks and twisted bands. We assume $D$ is reduced in the sense that it does not contain any nugatory crossings and that the sign of crossings in each band of $F$ is constant.
    When every band of $F$ is maximal (i.e. no disk is incident to exactly two bands) and contains at least one full twist, we call $F$ a twisted surface. A diagram $D$ is twisted if it determines such a surface via a checkerboard coloring, and a knot is twisted if it admits a twisted diagram. For instance, the standard diagram of the $P(a_1, a_2, \dots, a_n)$ pretzel knot is twisted so long as each $a_i$ satisfies $|a_i|>1$.

    Retract the spanning surface $F$ to a planar graph $\Gamma$, where vertices of $\Gamma$ represent disks of $F$ and edges represent bands of $F$. The edges of $\Gamma$ are weighted with the number of half-twists in the corresponding band of $F$. Denote the dual weighted planar graph by $\Gamma^\ast$, where each edge of $\Gamma^\ast$ inherits the weight of the corresponding edge of $\Gamma$. The surface $F$ is twisted if and only if all weights of $\Gamma$ are at least 2 in absolute value and $\Gamma$ has no vertices of valence one or two; $L$ is then a twisted knot. In  ~\cite{baader2019coxeter}, the authors prove the MRC for twisted knots, and they show that the bridge number of a twisted knot $L$ is equal to the number of vertices in $\Gamma^\ast$.

    Let $K$ be a twisted knot, let $v^*$ be the number of vertices in $\Gamma^\ast$, and let $v, e, f$ be the number of vertices, edges, and faces, respectively, of the planar embedding of $\Gamma$ induced by $F$. Then $f=v^*$ and $2=v-e+v^*$. Since $F$ retracts to $\Gamma$, then $\chi(F)=v-e$ and $2=\chi(F)+v^*$.  Since $v^*=\beta(K)$ by~\cite{baader2019coxeter}, we have that
    \begin{equation}\label{eq:chi-beta}
        \chi(F)= 2-\beta(K).
    \end{equation}

    If $F$ is orientable, then maximally compressing $F$ results in an essential Seifert surface, $F^c$, with genus at most the genus of $F$ and with one, non-meridional, boundary component. By Theorem \ref{NonMeridionalBounds}, $\beta(K)\leq 5$, or $d(K)\leq \frac{8g(F^c)-8}{5|\partial F^c|} + \frac{4}{5}+4\leq 2g(F)+\frac{16}{5}$. Since $d(K)$ is an integer, $d(K)\leq 2g(F)+3$. Since $g(F)=\frac{1}{2}(1-\chi(F))$, using Equation~\ref{eq:chi-beta} we conclude that 
    $d(K)\leq \beta(K)+2$. 

    If $F$ is non-orientable, we can use a similar construction. Let $H_F$ be a closed regular neighborhood of $F$ in $S^3$. So, $H_F$ is a handlebody of genus $1-\chi(\Gamma)=1-\chi(F)$. Since $K$ can be assumed to lie in the surface $\partial H_F$, we have that $\partial H_F$ meets a neighborhood of $K$ in exactly one annulus, denoted $A$. Hence $F':=\partial H_F \setminus int(A)$ is a compact, connected, orientable surface properly embedded in the exterior of $K$. Since $F'$ is obtained from $\partial H_F$ by deleting an annulus whose core is non-separating in $\partial H_F$, we obtain  $g(F')=g(\partial H_F)-1 =-\chi(F)$. Moreover, $F'$ has two boundary components. Note that $F'$ may be compressible. Maximally compressing $F'$ in the exterior of $L$ produces an essential surface $F''$ with genus at most  $-\chi(F)$ and with at most two boundary components. If $\beta(K)\leq 5$, then the theorem follows. So, we will assume $\beta(K)\geq 6$. By Theorem \ref{NonMeridionalBounds}, $\beta(K)(d(K)-4)
\leq \frac{8g(F'')-8}{|\partial F''|}+4.$ Since $\beta(K)\geq 6$, $d(K)\leq \frac{8g(F'')-8}{6|\partial F''|}+ \frac{14}{3}$.  Since $1\leq |\partial F''|\leq 2$, then $d(K)\leq \frac{8g(F'')-8}{6}+ \frac{14}{3}=\frac{4}{3}g(F'')+\frac{10}{3}.$ Since $g(F'')\leq -\chi(F)=\beta(K)-2$, then $d(K)\leq \frac{4}{3}\beta(K)+\frac{2}{3}$. 
\end{proof}

\section{Swimming with the Fishnets} \label{sec:limerick}

\noindent {\it Any link hanging down from a bridge}\\
\noindent {\it Can be a fishnet: jump in and --- sea! ---}\\
\noindent {\it Spin an icy unknot}\\
\noindent {\it Shivers get got}\\
\noindent {\it Swim in the unknot to catch MRC}\\

\section*{Acknowledgements}
\noindent
RB is partially supported by NSF grant DMS-2424734. AK and EP are partially supported by NSF grant DMS-2204349. We thank Michel Boileau for giving us feedback on a draft of this paper.

\bibliographystyle{alpha}
\bibliography{wirtinger}

\vspace{1cm}

Department of Mathematics and Statistics, California State University Long Beach\\
\noindent Long Beach, CA, 90840, USA\\

Department of Mathematics, University of Notre Dame\\
\noindent
Notre Dame, IN, 46556, USA\\

Department of Mathematics, University of Notre Dame\\
\noindent
Notre Dame, IN, 46556, USA\\
\end{document}